\documentclass[a4paper,twoside,11pt,pdftex]{article}

\usepackage{a4wide} 
\usepackage[utf8]{inputenc}
\usepackage[T1]{fontenc}
\usepackage{lmodern}

\usepackage{graphicx} 
\usepackage{verbatim} 
\usepackage{amsmath,amssymb,amsthm}
\usepackage{color}
\usepackage{enumerate}
\usepackage{array}
\usepackage{cite}

\usepackage{caption}
\usepackage{subfig}

\usepackage{tikz}

\usepackage{hyperref}
\definecolor{darkgreen}{rgb}{0,0.4,0}
\definecolor{BrickRed}{rgb}{0.65,0.08,0}
\hypersetup{colorlinks=true,linkcolor=blue,citecolor=red,filecolor=BrickRed,urlcolor=darkgreen}
\newcommand{\ct}{c}
\newcommand{\rt}{r}

\newcommand{\Rgf}{R}
\newcommand{\rgf}{r}

\newcommand{\Cgf}{C}
\newcommand{\cgf}{c}

\DeclareMathOperator\artanh{artanh}
\DeclareMathOperator\Cat{Cat}

\newcommand{\LandauO}{\mathcal{O}}
\newcommand{\Landauo}{o}

\newcommand{\Cc}{\mathcal{C}}
\newcommand{\Dc}{\mathcal{D}}

\newcommand{\Ic}{\mathcal{I}}
\newcommand{\Kc}{\mathcal{K}}
\newcommand{\Rc}{\mathcal{R}}
\newcommand{\Sc}{\mathcal{S}}
\newcommand{\Tc}{\mathcal{T}}
\newcommand{\C}{\mathbb{C}}
\newcommand{\N}{\mathbb{N}}

\newcommand{\Z}{{\mathbb Z}}

\newtheorem{theo}{Theorem}[section]
\newtheorem{lemma}[theo]{Lemma}
\newtheorem{prop}[theo]{Proposition}
\newtheorem{coro}[theo]{Corollary}
\newtheorem{definition}[theo]{Definition}
\newenvironment{example}[1][]{\refstepcounter{theo} \medskip \noindent \textbf{\textit{Example \thetheo #1:}} }{ \hfill $_{\blacksquare} $\\}
\newenvironment{remark}[1][]{\refstepcounter{theo}\medskip \noindent\textbf{\textit{Remark \thetheo #1:}} }{ \hfill $_{\blacksquare} $\\ }

\newcommand{\OEIS}[1]{\text{\href{https://oeis.org/#1}{{\small \tt #1}}}}

\newtheoremstyle{conjecture}{}{}{\it}{}{\color{purple}\bfseries}{}{ }{}
\theoremstyle{conjecture}

\graphicspath{{./}{pics/}}
\makeatletter
\def\input@path{{./}{pics}}  
\makeatother

\date{}

\begin{document}

\author{Antoine Genitrini\thanks{Sorbonne Universit\'es, UPMC Univ Paris 06, CNRS, LIP6 UMR 7606, 4 place Jussieu 75005 Paris.
	\url{Antoine.Genitrini@lip6.fr}}
 \and Bernhard Gittenberger\thanks{Technische Universit\"at Wien, Wiedner Hauptstra\ss e 8-10/104, 1040 Wien, Austria.
	\url{{Bernhard.Gittenberger, Michael.Wallner}@tuwien.ac.at}}
 \and Manuel Kauers\thanks{Institute for Algebra, Johannes Kepler University, Altenberger Strasse 69, 4040 Linz, Austria. \url{Manuel@Kauers.de}} \and Michael Wallner$^\ddag$}

\title{Asymptotic Enumeration of Compacted Binary Trees of Bounded Right Height\thanks{This research was supported by the Austrian Science Fund (FWF) grant SFB F50-03.}}

\maketitle

\let\thefootnote\relax\footnotetext{{\textcopyright}~2020. This manuscript version is made available under the CC-BY-NC-ND 4.0 license \url{http://creativecommons.org/licenses/by-nc-nd/4.0/}}

\begin{abstract}
A compacted binary tree is a graph created from a binary tree such that repeatedly occurring subtrees in the original tree are represented by pointers to existing ones, and hence every subtree is unique. Such representations form a special class of directed
acyclic graphs.  We are interested in the asymptotic number of compacted trees of given size,
where the size of a compacted tree is given by the number of its internal nodes. Due to its
superexponential growth this problem poses many difficulties. Therefore we
restrict our investigations to compacted trees of bounded right height, which is the maximal number of edges going to the right on any path from the root to a leaf. 

We solve the asymptotic 
counting problem for this class as well as a closely related, further simplified class. 

For this purpose, we develop a calculus on exponential generating functions for compacted trees of bounded
right height and for relaxed trees of bounded right height, which differ from compacted trees by
dropping the above described uniqueness condition. 
This enables us to derive a recursively defined sequence of differential equations
for the exponential generating functions. The coefficients can then be determined by performing a singularity
analysis of the solutions of these differential equations. 

Our main results are the computation of the asymptotic numbers of
relaxed as well as compacted trees of bounded right height and given size, when the size tends to
infinity. 

\end{abstract}

\clearpage

\section{Introduction}
\label{sec:intro}

Most trees contain redundant information in form of repeated occurrences of the same subtree\footnote{In the rest of the paper,
a subtree of a given tree contains a root and all its descendants in the original tree. Such a substructure
is sometimes called a fringe subtree.}.
In order to get an efficient representation in memory,
these trees can be compacted by representing each occurrence only once. 
The removed subtrees
are replaced by pointers which link to the shared subtree.
Such structures are classically named as \emph{directed acyclic graphs} or short as \emph{DAGs}.

Flajolet \emph{et al.}, in their extended abstract~\cite{flss90}, analyzed in detail the gain in memory
of the compaction. Some proofs have been omitted and have not been stated later.
This gap was closed in~\cite{bousquet2015xml}, where the framework was extended to other DAG structures
and analyzed in the context of XML compression.
Furthermore, Ralaivaosaona and Wagner extended in \cite{RalaivaosaonaWagner2015Repeated} the analysis of the gained memory to simply generated families of trees.

The latter two papers on the quantitative analysis of the compaction process,
studied the transformation of a given set of trees of given size
to the set of compacted trees in order to determine the average rate of compaction.
We focus on a different aspect, namely the enumeration problem of compacted binary trees.
On the one hand, enumerating combinatorial structures is important if one wants to understand shape characteristics
of large random structures or for uniform random generation of those structures. On the other
hand, the enumeration of particular classes of DAGs is in general a difficult problem which
requires the extension of combinatorial methodology and is therefore interesting in its own right. 

One of the difficulties in the enumeration of compacted binary trees lies in the fact that a
compacted binary tree of size $n$ could arise from a binary tree whose size belongs to the whole interval $n,\dots, 2^n$. Thus, a brute-force approach is hopeless. 

The first papers about the enumeration of DAGs appeared in the 1970's. Robinson
presented two distinct approaches~\cite{Robinson70, Robinson77}
based either on some inclusion-exclusion method or on P\'olya's enumeration theory~\cite{Polya37}. 
Combining combinatorial arguments and analytic methods the asymptotic number of labeled DAGs was
determined in \cite{BRRW86}, for connected structures then in \cite{BR88}. 
The first investigation of shape
parameters seems to go back to McKay \cite{McKay89}. Recently, enumeration results for many
particular classes of DAGs can be found in the literature, see for instance \cite{BGG11, BGGG17, 
BGGJ13, GS05, liskovets2006exact, MSW15, St03, St04, Wa13}, as well as investigations on the (random) generation of particular DAGs, see \cite{ADK05, CD12, MDB01, MP04}. 

A now classical way for enumeration is the use of generating functions.
In this context, precisely for labeled structures (see paper~\cite{robinson1973dags}),
Robinson designed generating functions of a very particular nature to solve an asymptotic
counting problem concerning DAGs. The classical types of generating functions like ordinary and
exponential ones were not suited for the problem. 

We are facing the same problem in the enumeration of compacted trees. Indeed, due to the fact that
compacted trees are unlabeled combinatorial structures, which are moreover closely related to
plane trees, a treatment with ordinary generating functions will be the first choice. However, the 
fast growth of the counting sequence requires the use of exponential generating functions. In
order to be able to get asymptotic results, we will confine ourselves to certain subclasses of the
class of compacted trees as well as some related classes by relaxing certain conditions. Moreover,
we will develop a calculus for exponential generating functions designed for these classes.
Bounding the right height of our DAGs leads to a sequence of D-finite functions (see
\cite{KauersPaule11,stan99} for introductions to the subject) for which it is
possible to analyze their differential equations and obtain finally our main result. Likewise, 
in other enumeration problems for particular classes of DAGs bounding a certain parameter turned
intractable recurrences into D-finite ones. Examples are the enumeration of certain classes of 
lambda-terms \cite{BGG11, BGGJ13, BGGG17} or increasing series-parallel DAGs \cite{BDGP17}.

\subsubsection*{Plan of the paper}

Our combinatorial structures are based on the fundamental properties of the compaction procedure.
We will first analyze some properties of this classical procedure (linked to the \emph{common
subexpression problem}) in Section~\ref{sec:creating}.

Then we will define the basic concepts and state our main results in Section~\ref{sec:mainres},
see Theorems~\ref{theo:relaxededmain} and \ref{theo:compactedmain}. 

Some basic observations concerning the structure of compacted trees will then be presented in
Section~\ref{sec:struct}. 

These will help us to state a combinatorial and (most importantly) recursive specification
of the problem in Section~\ref{sec:recurrence}. A further important result is the derivation of 
\emph{a recurrence relation for the number of compacted binary trees}, see Theorem~\ref{theo:comprecursion}.
This recurrence is not classical at all, and we are not able to solve it explicitly.

Due to this fact, we follow yet a different approach in the remaining part of this work:
We will use exponential generating functions to model our problem, as the superexponential growth
rate of the counting sequence suggests, though we are dealing with
unlabeled combinatorial structures. Therefore, a new calculus translating certain set operations for
classes of compacted trees into algebraic operations of exponential generating functions will be
developed in Section~\ref{sec:operations}. 

Section~\ref{sec:relaxed} is devoted to a simplified problem, the study of the counting problem of
\emph{relaxed binary trees}. These DAGs are in a sense compacted trees where the restriction of
uniqueness on the subtrees is relaxed. In particular, compacted trees are a subset of relaxed
binary trees. With the same methods as used on compacted trees we are able to derive a recurrence
relation. However, this recurrence relation is as difficult as the first one for compacted trees.

A natural constraint for compacted trees seems to bound some specific depth limit, the so-called
right height. This is the maximal number of edges directed to the right which appear on any path
from the root to a leaf. In Section~\ref{sec:relaxed}, the calculus developed in
Section~\ref{sec:operations} enables us to derive a differential equation for the generating
function of relaxed trees for each bound $k$ on the right height. This sequence of D-finite 
differential equations follows a rather explicit recursive scheme, presented in
Theorem~\ref{theo:Dkprop} which allows us to analyze the dominant singularities of the solutions
of the differential equation for any $k$. Eventually, this strategy is successful and we are able to
determine the asymptotic number of relaxed binary trees of bounded right height. 

Finally, in Section~\ref{sec:compacted} we modify the results of the previous section to cover the
case of compacted trees as well. Again, we derive a sequence of D-finite differential equations where, as in
Section~\ref{sec:relaxed}, the dominant singularities of the generating function are regular
singularities of the differential equation. This allows us to extract the asymptotic behavior
of the counting sequence, which contains irrational powers of $n$. The necessary information is
directly extracted from the differential equations. Except for the first few, they do not have 
closed-form solutions.

\section{Creating a compacted tree}
\label{sec:creating}

Many problems in computer science and computer algebra involve redundant information.
A strategy to save memory is to store every instance only once and to point to already existing
instances, whenever an instance appears repeatedly.  
In \cite[Proposition~1]{flss90} a compression algorithm was presented, and it was shown that for a given tree of size $n$, its compacted form can be computed in expected time $\LandauO(n)$.
However, such procedures have been known since the $1970$'s (see~\cite{flss90,DowneySethiTarjan1980variations} and especially the ``value-number method'' in compiling~\cite[Section~6.1.2]{aho1986compilers}).
Figure \ref{fig:procUID} shows this procedure, which follows a top-down decomposition scheme
(i.e.~post-order traversal) of labeled binary trees.
Every node (or actually the subtree whose root is the respective node)
is associated with a ``unique identifier'' (\emph{uid}). 
Two subtrees are equivalent if and only if the uid's are the same. 

\begin{figure}[ht]
\noindent\makebox[\linewidth]{\rule{\textwidth}{0.4pt}}\\[-0.8\baselineskip]
\noindent\makebox[\linewidth]{\rule{\textwidth}{0.4pt}}
\texttt{function UID(T : tree) : integer;}\\[-0.5\baselineskip]
\noindent\makebox[\linewidth]{\rule{\textwidth}{0.4pt}}%
\vspace{-0.5\baselineskip}
\begin{verbatim}
global counter : integer, Table : list;
begin 
   if  T  =  nil 
      then  return(O); 
      else 
         triple  :=  <root(T),UID(left(T)),UID(right(T))>; 
         if  Found(triple,Table) 
            then  return(value_found); 
            else  counter  :=  counter+l; 
               Insert  pair  (triple,counter)  in  Table; 
               return(counter); 
         fi 
   fi 
end 
\end{verbatim}
\vspace{-1\baselineskip}
\noindent\makebox[\linewidth]{\rule{\textwidth}{0.4pt}}\\[-0.8\baselineskip]
\noindent\makebox[\linewidth]{\rule{\textwidth}{0.4pt}}
\caption{\small The \texttt{UID} procedure from \cite[Fig.~2]{flss90} which computes
``unique identifiers'' for all (fringe) subtrees of a given binary tree $T$. 
It is assumed that \texttt{counter} is initially set to $0$. 
Table is a global list that maintains associations between triples and already 
computed \texttt{UID}'s; it is also initially empty.
The function \texttt{root(T)} extracts the label of the root of tree $T$. }%
\label{fig:procUID}%
\end{figure}

We now give an example of the behavior of the procedure for an arithmetic expression. 

\begin{example}
 Consider the labeled tree necessary to store the arithmetic expression
 \texttt{(* (- (* x x) (* y y)) (+ (* x x) (* y y)))} which represents
 $(x^2-y^2)(x^2+y^2)$. The ``Table'', built by the \texttt{UID} procedure, contains
	\begin{align*}
		((x,0,0), 1),   && ((y,0,0), 3),   && ((-,2,4), 5),   && ((\times,5,6), 7), \\
		((\times,1,1), 2),   && ((\times,3,3), 4),   && ((+,2,4), 6),
	\end{align*}
	and the tree in its full and compacted version is shown in Figure \ref{fig:UIDex}.
\end{example}
	
\begin{figure}[ht]%
	\begin{center}
	\includegraphics[width=0.9\textwidth]{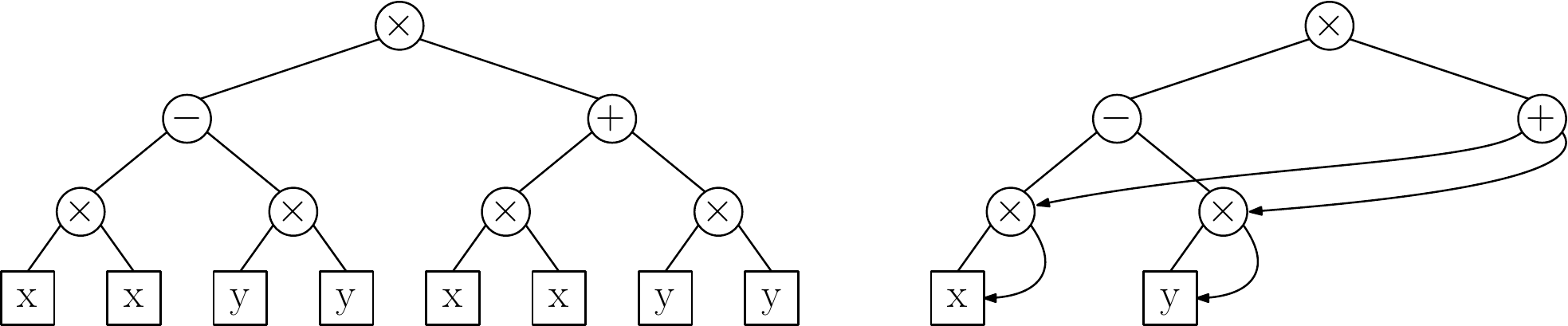}%
	\caption{\small Tree and compacted tree associated with 
	\texttt{(* (- (* x x) (* y y)) (+ (* x x) (* y y)))} computed by the 
	\texttt{UID} procedure from Figure \ref{fig:procUID}.}%
	\label{fig:UIDex}%
	\end{center}
\end{figure}

Motivated by this procedure, based on a post-order traversal of the tree,
we define an ad hoc DAG-structure, which we call a \emph{compacted binary tree},
that encodes the result of the compaction of the tree.
The trees under consideration are \emph{full binary} in the sense that their
nodes have either 0 or 2 children.
Furthermore, in the definition we refer to subtrees: a \emph{fringe subtree} or short \emph{subtree}
is the tree which corresponds to a node and all its descendants. 
In this paper we only consider such subtrees.

\begin{definition}
	\label{def:compactedbinarytree}
	A \emph{compacted binary tree} is a DAG computed by the \texttt{UID}
	procedure from a given full binary tree.
	Every edge leading to a subtree that has already been seen
	during the traversal is replaced by a new kind of edge, a \emph{pointer},
	to the already existing subtree. The \emph{size} of the compacted
	binary tree is defined by the number of its internal nodes. 
\end{definition}	 

In the sequel we will only consider full binary trees and their compacted forms.
Thus, the term \emph{compacted trees} means compacted binary trees. 
In Figure~\ref{fig:compacted_trees_n123}, we represent all compacted trees of size $0,1$, and $2$.


\begin{figure}[ht]
	\centering
	\includegraphics[width=0.7\textwidth]{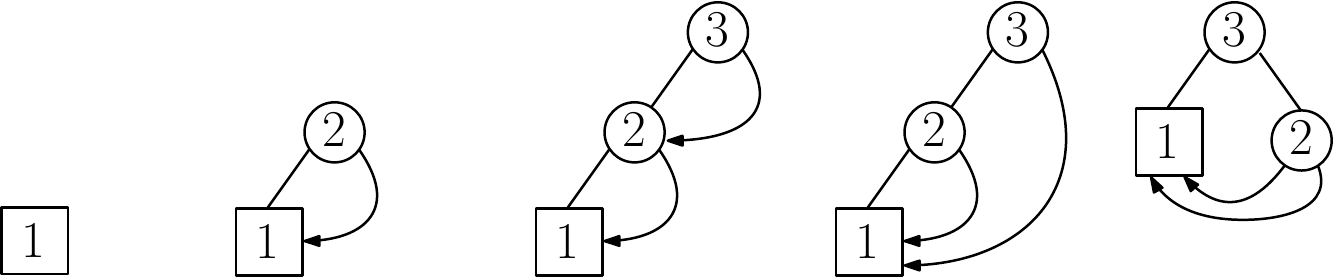}%
	\caption{\small All compacted trees of size $n=0,1,2$. The labels in the nodes are the uids of
the corresponding subtrees. But note: The labels are not belonging to the combinatorial objects.
Compacted trees are unlableled graphs.}
	\label{fig:compacted_trees_n123}
\end{figure}

The subclass of DAGs we are interested in is strongly influenced by properties of trees. 
In particular, compacted trees are connected and plane. The out-degree\footnote{For the terms out-
and in-degree, source, sink, and so on, we interpret an undirected edge as directed away from the
root, in accordance with a node-child relation.} of each node is equal to $2$,
except for the unique sink (leaf) for which it is $0$.
Furthermore, there is a unique source, which is the root.

The latter properties are induced by the full binary tree structure.
Next, we treat the specific properties of the \texttt{UID} procedure.
The result of the algorithm strongly depends on the chosen traversal.
In this case the post-order traversal is used --
but one could also consider a different one.
There are two important observations.
First of all, it has an important consequence on the pointers:
\begin{prop}
\label{prop:pointerrestriction}
In a compacted tree the pointers only point to previously discovered trees.
\end{prop}

In other words, the ordering imposed by the traversal restricts the possible choices of the pointers. 

\begin{definition}
	For any compacted tree of size $n$, the \emph{spine} is the structure (with $n$ nodes)
	obtained from the compacted tree by deleting all pointers and the leaf.
\end{definition}
In the Figure~\ref{fig:spine}, from left to right, we see a compacted tree (without details on the pointers) and its spine. Furthermore, every distinct subtree is stored only once.
In terms of the corresponding compacted trees this translates into
uniqueness of every subtree.

\begin{figure}[htb]
	\centering
	\includegraphics[width=0.2\textwidth]{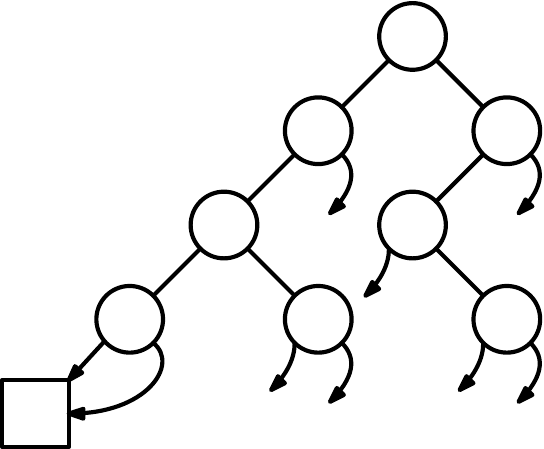}
	\qquad \qquad \qquad
	\includegraphics[width=0.2\textwidth]{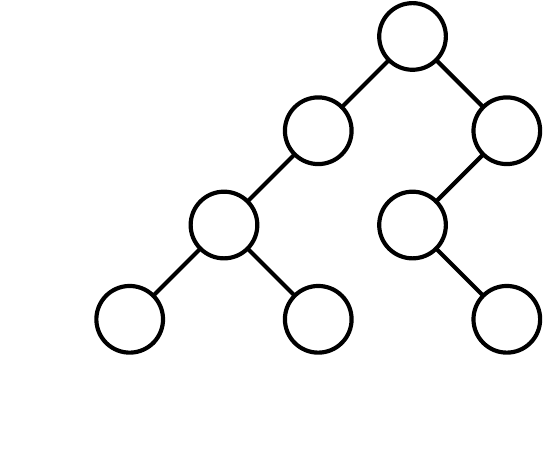}
	\caption{\small A compacted tree and its spine.}
	\label{fig:spine}
\end{figure}

\section{Main results}
\label{sec:mainres}

Before being able to state our main results we have to define further combinatorial classes.
Indeed, the uniqueness condition for compacted trees caused some difficulties in their
enumeration. So, we will first analyze a simpler class where we drop this condition. 

\begin{definition}
A \emph{relaxed compacted binary tree} (short \emph{relaxed binary tree}, or just \emph{relaxed
tree}), of size $n$ is a directed acyclic graph consisting
of a binary tree with $n$ internal nodes, one leaf, and $n$ pointers.
It is constructed from a binary tree of size $n$, where the first leaf
in a post-order traversal is kept and all other leaves are replaced by pointers.
These links may point to any node that has already been visited by the post-order traversal.

Obviously, the notion of spine adapts to the class of relaxed trees.
\end{definition}

In fact, let us give another way to interpret compacted trees: 
compacted trees are relaxed trees with the restriction that all nodes in the spine are the roots of unique subtrees of the full tree.
Note that this condition does not hold for all relaxed trees. In particular compare Figure~\ref{fig:relaxed_tree_3} for the smallest relaxed tree which is not a compacted tree. 

\begin{figure}[ht]
	\centering
	\includegraphics[width=0.55\textwidth]{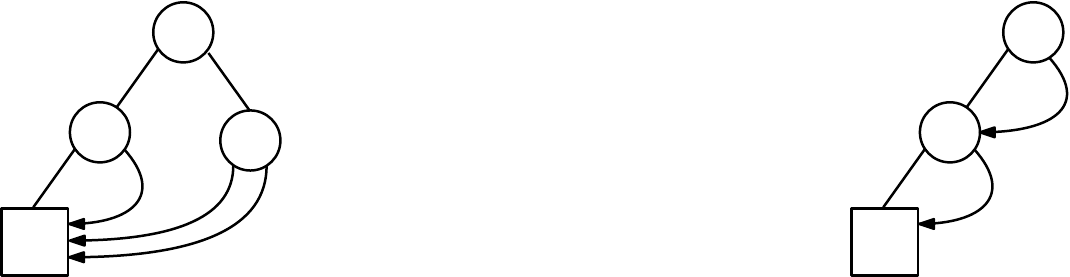}
	\caption{\small Left: the smallest relaxed tree that is not a compacted tree; right: the corresponding unique compacted tree.}
	\label{fig:relaxed_tree_3}
\end{figure}

The asymptotic enumeration of relaxed trees is still too complicated. We will derive recurrence
relations for their counting sequence as well as for the counting sequence of compacted trees. In
order to obtain asymptotic results, we restrict the right height. 

\begin{definition}
For any relaxed tree, we define its \emph{right height} to be the maximal number of right edges
on any path from the root to another node in the spine (of the relaxed or compacted tree under 
consideration).
The \emph{level} of a node is the number of right edges on the path from the root to this node. 
\end{definition}

Figure~\ref{fig:rightheight} introduces an example and a natural way of representing a relaxed tree
in order to emphasize these notions. It proves convenient to rotate the trees by $45$ degrees.
\begin{figure}[htb]
	\centering
	\includegraphics[width=0.2\textwidth]{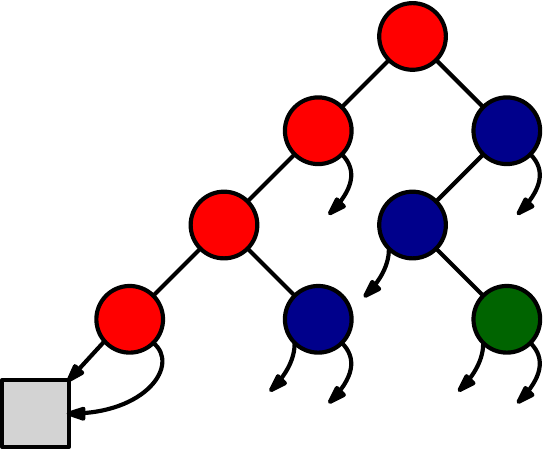}
	\qquad \qquad \qquad
	\includegraphics[width=0.26\textwidth]{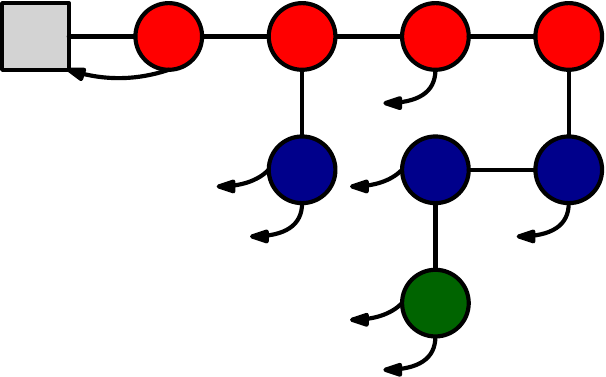}
	\caption{\small A compacted tree with right height~$2$.
	Nodes of level~$0$ are colored in red, nodes of level~$1$ in blue,
	and the node of level~$2$ in green. }
	\label{fig:rightheight}
\end{figure}

Bounding the right height defines a sequence of classes which follows a recursive construction
principle. We will eventually exploit this structure and obtain our main results, 
the asymptotic number of relaxed trees with $n$ internal nodes and the analogous result for
compacted trees.

\begin{theo}[Asymptotics of relaxed trees with bounded right height]
	\label{theo:relaxededmain}
	The number $r_{k,n}$ of relaxed trees with right height at most $k$ is for $n \to \infty$ asymptotically 
	equivalent to
	\begin{align*}
		r_{k,n} \sim \gamma_k  n! \left(4 \cos\left(\frac{\pi}{k+3}\right)^2\right)^{n}  n^{-k/2},
	\end{align*}
	where $\gamma_k$ is a positive constant which is independent of $n$.
\end{theo}

\begin{theo}[Asymptotics of compacted trees with bounded right height]
	\label{theo:compactedmain}
	The number $c_{k,n}$ of compacted trees with right height at most $k$ is for $n \to \infty$ asymptotically equivalent to
	\begin{align*}
		c_{k,n} \sim \kappa_k   n! 
		\left(4 \cos\left(\frac{\pi}{k+3}\right)^2\right)^{n} 
		n^{- \frac{k}{2} - \frac{1}{k+3} - \left(\frac{1}{4} - \frac{1}{k+3}\right)\cos\left(\frac{\pi}{k+3} \right)^{-2}} 
						,
	\end{align*}
	where $\kappa_k$ is a positive constant which is independent of $n$.
\end{theo}

Therefore, we can also answer the question (at least asymptotically) of how many relaxed trees are actually compacted trees.
Combining Theorems~\ref{theo:relaxededmain} and \ref{theo:compactedmain} we get the following result.

\begin{coro}[Proportion of compacted among relaxed trees]
	\label{coro:compamongrelaxed}
	Let $c_{k,n}$ ($r_{k,n}$) be the number of compacted (relaxed) binary trees with right height at most $k$. Then, for $n \to \infty$ we have
	\begin{align*}
		\frac{c_{k,n}}{r_{k,n}} \sim 
			\frac{\kappa_k}{\gamma_k} n^{- \frac{1}{k+3} - \left(\frac{1}{4} - \frac{1}{k+3}\right)\frac{1}{\cos^2\left(\frac{\pi}{k+3} \right)}}
			= \Landauo\left(n^{-1/4} \right)
			.
	\end{align*}
\end{coro}

Thus, the number of compacted trees among relaxed trees for large $n$ is negligible. This result quantifies the restriction of uniqueness of subtrees in compacted trees.

\section{On the structure of compacted trees}
\label{sec:struct}

In this section we will discuss some basic observations concerning the structure of compacted trees. 
First note that pointers may point to nodes lying outside the subtree of the pointer's
start node (compare with Figures~\ref{fig:UIDex} and \ref{fig:compacted_trees_n123}). Such subtrees of compacted trees
cannot be compacted trees themselves. For this reason, we define the concept of \emph{c-subtrees}.

\begin{definition}
	 A \emph{c-subtree} is the subgraph of a compacted tree induced by a node and all its
descendants. 	 	 
	 A \emph{cherry} is a c-subtree where both children of the root are pointers. 
\end{definition}

A cherry is, in a sense, the ``minimal'' construction to create a new (unique) subtree.
It consists of a node and two pointers, which point to already found $c$-subtrees during the
traversal process. An example is given in Figure~\ref{fig:compacted_trees_n123}: In the rightmost tree, the $c$-subtree
with the root node labeled by $2$ is a cherry. Such a cherry is not a compacted tree in the
sense of Definition~\ref{def:compactedbinarytree}, as the root node has two pointers which point
to an external structure. It represents, however, a subtree and it corresponds in a unique way 
to the compacted tree of this subtree. 
The only compacted tree of size~$1$ is also given in the same figure.

With this terminology we are able to analyze some aspects of the DAG-structure of compacted trees.  
First, we look at the spine. 

\begin{lemma}
	\label{lem:deletepointers}
	The spine of a compacted tree of size $n$ is a binary tree of size $n$. 
\end{lemma}

\begin{proof}
	Obviously, by deleting the leaf and the pointers we get a rooted, acyclic graph.
	It remains to show that this graph is connected. 
	Assume that there exists a pointer which is the only connection between two parts of the compacted tree. 
	By the \texttt{UID} procedure a pointer corresponds to a multiple occurrence of a subtree.
	Therefore we get a contradiction, as this subtree must already exist in the tree and is,
	therefore, connected with the root via internal edges. 
\end{proof}

Let us remark that the tree structure of a spine is binary in the sense that its nodes are either of out-degree $2$,
 $1$ (with two possibilities, either with a left child or with a right child), or $0$.

\begin{prop}
	\label{prop:characterizingcompactedtrees}
	From any binary tree of size $n$, we can build a compacted tree of size $n$, with the following operations:
	\begin{enumerate}
		\item Add a leaf as left child of the leftmost node of the binary tree.
		\item Add pointers to every node such that every node except the leaf has out-degree $2$.
		\item\label{item:cherryuniqueness} Let the pointers point to internal nodes which are in post-order traversal
			before the root node (under consideration) such that the corresponding subtree is unique (not already existing).
	\end{enumerate}
	Every compacted tree of size $n$ can be constructed this way. 
\end{prop} 

\begin{proof}
	A simple way to build a compacted tree by using the spine is the following one.
	Add the leaf to the leftmost node of the binary tree. Then traverse the binary tree
	by using the post-order traversal. Each time one meets a node with out-degree less than $2$
	one adds $1$ or $2$ pointers such that the uniqueness condition is not violated (e.g.~to the last node that has been visited).
	Thus, starting from a DAG generated from the above operations, decompacting and compacting it again with the \texttt{UID} procedure one arrives at the same structure.

	The last statement is obvious, since every compacted tree can be reconstructed from its spine
using only the operations listed above. By Lemma~\ref{lem:deletepointers} the spine has the same
size as the compacted tree. 
\end{proof}
	
The advantage of the previous proposition is that it gives us an alternative construction of compacted trees bypassing the \texttt{UID} procedure.
Starting from a binary tree, we can construct several compacted trees by enriching this binary tree.
So the function mapping compacted trees to its spine is not one-to-one.

A key observations is that cherries are the fundamental structures that guarantee the uniqueness 
of c-subtrees.
Indeed, if a cherry violates the condition implicit in the third operation listed in Proposition~\ref{prop:characterizingcompactedtrees},
 the structure is not a compacted tree according to our definition, but only 
a relaxed tree.

A different explanation why cherries are the crucial objects for uniqueness
comes from the property that the compaction procedure generates an increasing set of elements,
i.e.~already seen subtrees. 
Here we mean that the next element is constructed by a new internal node and previous,
already built elements. 
In particular, the first element is always a leaf, the second one is always
an internal node with two leaves as children (a ``classical cherry''). 
Then, as a third element one has an element with a new internal node and a cherry as its left
child, or as its right child, or on both sides. How will further elements be built 
such that the uniqueness property is maintained? Let us focus on the bad ways to do so,
\emph{i.e.}, we ask: What is forbidden? 
There are two cases according to the type of the current node (in the post-order traversal of the
tree):

\begin{itemize}
	\item The current node is a cherry: The only forbidden way to place the two pointers is 
choosing an already generated subtree and letting the two pointers of the cherry point to the
children of the subtree. Note that the children of an already generated subtree must have been
generated before. Thus, for any already generated subtree there is one forbidden configuration for
the placement of the pointers. 
	\item The current node is not a cherry: In this case at least one edge is not a pointer.
		But then it can easily be seen by induction (on the size of c-subtrees) that the subtree of the corresponding
child is unique when assigning pointers during a post-order traversal: If the pointer is the left
edge, the right subtree has not been processed yet when the decision for the pointer's target is 
about to come. 
Otherwise, the left subtree is unique and its building (placing of all its
pointers) ends just before the pointer being the right edge is processed. 
Hence, there is no restriction on placing the pointer since the current node will always generate an new subtree due to the unique first appearance of the subtree of one of its children. 
\end{itemize}

This idea will be picked up in the next section and used to derive a recurrence relation for the
number of compacted trees of size $n$. Besides, it shows that we have to be careful only when
dealing with nodes having two pointers (see Section~\ref{sec:compacted}). 

\section{Compacted trees of unbounded right height}
\label{sec:recurrence}

Using the properties stated in the last section for compacted trees, we are now
able to exhibit a combinatorial recurrence based on a decomposition of
the structures under consideration.

\subsection{A recurrence relation for compacted trees}

Let $\ct_n$ be the number of compacted binary trees of size $n$.
Recall that Figure~\ref{fig:compacted_trees_n123} showed all compacted trees
of size $0,1$ and $2$. The first few terms of the sequence are given by
\begin{align*}
	\left( \ct_n \right)_{n \geq 0} &= 
		\left( 
			1, 1, 3, 15, 111, 1119, 14487, 230943, 4395855, 97608831, 
		\ldots \right).	
\end{align*}

This sequence is found as sequence \OEIS{A254789} in Sloane's Online Encyclopedia of Integer Sequences\footnote{\href{http://www.oeis.org}{www.oeis.org}}.
Let us mention that it appeared independently online during our work on this problem. 
In this section we solve the counting problem by deriving the first defining recurrence relation.

Suppose that we perform a post-order traversal on a tree and that already $p$ c-subtrees have been
discovered. Then the current node is the root of another c-subtree. 
Let $\Gamma_{n,p}$ denote the class of all c-subtrees
of size $n$ that may show up as such a c-subtree. Then we may think of the already
compacted subtrees as an external pool of trees where our pointers can point to additionally when
continuing our traversal. For an illustration see Figure~\ref{fig:pool}.
Note that the leaf is always part of this pool but not counted,
and all subtrees in the pool must be constructed out of elements from the pool.
In this sense the pool is closed in itself, and its evolution in the compaction procedure is an increasing sequence of sets.

\begin{figure}[htb]
	\centering
	\includegraphics[width=0.32\textwidth]{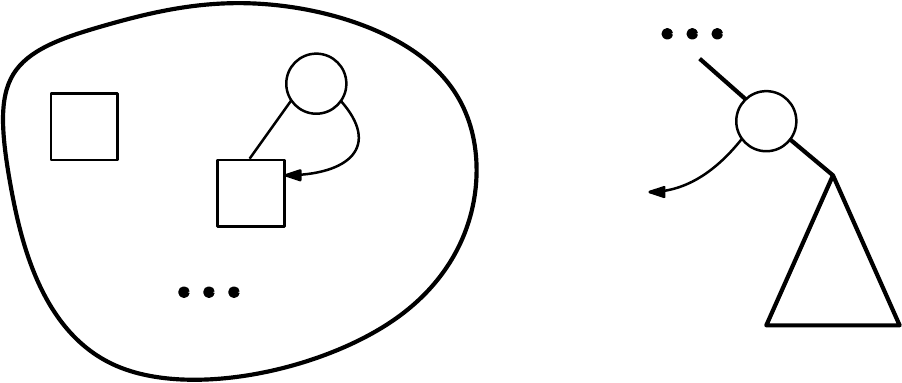} \qquad \qquad \qquad \qquad
	\includegraphics[width=0.3\textwidth]{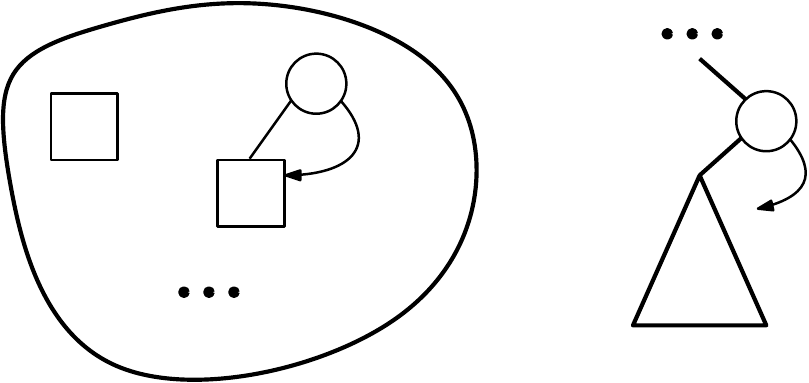}
	\caption{\small The two cases of the pool construction of Theorem~\ref{theo:comprecursion}. The pool (circled elements) represents the already visited c-subtrees the pointers may point to. In the second case it may also point to the c-subtrees of the left sibling.}
	\label{fig:pool}
\end{figure}

We define the size of the pool to be the number of distinct subtrees 
with at least one internal node. 
Thus, the pool for the trees in $\Gamma_{n,p}$ has size $p$ and consists of $p+1$ distinct 
c-subtrees. This artificially looking convention will simplify the later analysis.

\begin{theo}
	\label{theo:comprecursion}
	Let $n,p \in \N$, and $\Gamma_{n,p}$ as above. Moreover, we denote the cardinality of
$\Gamma_{n,p}$ by $\gamma_{n,p}$. Then 
	\begin{align}
		\gamma_{n+1,p} &= \sum_{i=0}^{n} \gamma_{i,p} \; \gamma_{n-i,p+i}, \qquad \text{ for } n \geq 1, \label{eq:recCnp}\\
		\gamma_{0,p} &= p+1, \label{eq:recCnpI1}\\
		\gamma_{1,p} &= p^2 + p + 1. \label{eq:recCnpI2}
	\end{align}
\end{theo}

\begin{proof}
An element of $\Gamma_{n,p}$ consists of $n$ internal nodes connected by $n-1$ internal edges. The remaining $n+1$ edges of the compacted binary tree are pointers (the possible edge to the leaf may be interpreted as a pointer). These must be chosen in such a way that no subtree is generated twice. Additionally, they may point either to a c-subtree of the pool or to a c-subtree of its left sibling, see Figure~\ref{fig:pool}. The second condition is due to the post-order traversal of the tree by the \texttt{UID} procedure.  
	
	Now we can give a recursive decomposition of such trees. Let $t$ be a c-subtree with $n+1$
nodes and a pool of size $p$. The root of $t$ has a left and a right subtree attached to $i$ and
$n-i$, (for $i=0,\ldots,n$) internal nodes, respectively. Note that every internal node also
represents a c-subtree. For the left child the pool remains the same as for its parent. However,
for the right child the pointers may additionally point to c-subtrees of its left sibling. Hence,
the pool is increased by the size of its left sibling. These considerations directly give Equation~\eqref{eq:recCnp}.
	
	Next, let us consider the initial conditions \eqref{eq:recCnpI1} and \eqref{eq:recCnpI2}. The c-subtrees with no internal nodes can be interpreted as pointers. These may point to any element of the pool, hence $\gamma_{0,p} = p+1$. 
	
	The c-subtrees with $1$ internal node are cherries whose both children are not internal nodes. Hence, they consist either of two pointers or of a leaf and a pointer. As the pool always contains a leaf, it is sufficient to consider the first case. Then these two pointers have $p+1$ possibilities each to point at. Among these $(p+1)^2$ cases are $p$ which must be excluded as they are the ones already found in the pool. Note that these can be recreated by letting the pointers point to the same children as the ones found in the pool. Hence, we get 
	\begin{align*}
		\gamma_{1,p} &= (p+1)^2-p = p^2+p+1. \qedhere
	\end{align*}
\end{proof}

\begin{coro}
The number $\ct_n$ of compacted trees of size $n$ is equal to $\gamma_{n,0}$.
\end{coro}
Obviously, by Theorem~\ref{theo:comprecursion} the numbers $\gamma_{n,0}$ depend on the numbers $\gamma_{k,p}$
for all $0 \leq k \leq n$ and all $0 \leq p \leq n$. Thus their computation is cubic in time and quadratic memory.

Next, let us state a simplified problem,
which also proves very difficult to solve, but is not as technical.

\subsection{A recurrence relation for relaxed compacted trees}
\label{sec:relaxedproblemdef}

Let $\rt_n$ be the number of relaxed trees of size $n$. The first few terms of the sequence are given by
\begin{align*}
	\left( \rt_n \right)_{n \geq 0} &= 
		\left( 
			1, 1, 3, 16, 127, 1363, 18628, 311250, 6173791, 142190703, 
		\ldots \right).
\end{align*}
This sequence is given by the sequence \OEIS{A082161} in the OEIS. 
The latter counts the number of deterministic completely defined initially connected
acyclic automata with $2$ inputs and $n$ transient unlabeled states
and a unique absorbing state, see~\cite{liskovets2006exact}.
The bijection of these structures to our (enriched) trees is obvious,
by traversing relaxed trees from the root to the leaf. 
We remark that the asymptotic behavior of the number of such structures seems not to be known.

Let $\delta_{	n,p}$ be the number of relaxed c-subtrees of size $n$ and a pool of size $p$.
We directly get a recurrence relation for these numbers, that is directly linked
to the one for $(\gamma_{n,p})_{n,p \in \N}$:
\begin{coro}
	\label{coro:relaxedrecurrence}
	Let $n,p \in \N$, then
	\begin{align}
		\delta_{n+1,p} &= \sum_{i=0}^{n} \delta_{i,p} \; \delta_{n-i,p+i}, \qquad \text{ for } n \geq 1, \label{eq:recRnp}\\
		\delta_{0,p} &= p+1. \label{eq:recRnpI1}
	\end{align}
	The number of relaxed trees of size $n$ is equal to $\delta_{n,0}$.
\end{coro}
\begin{proof}
	This is a direct consequence of Theorem~\ref{theo:comprecursion} and the fact that we dropped the uniqueness restriction enforced by \eqref{eq:recCnpI2}.
\end{proof}

Note that the nature of the recurrence relation did not change compared to the one of the compacted case.
Unfortunately, we were not able to find an explicit solution, or to continue from here. 
However, using our main results on trees of bounded right height we are able to determine the asymptotic growth of this sequence.

\subsection{The asymptotic growth of unbounded compacted and relaxed trees}

In order to better understand the asymptotic growth of compacted trees we first consider some simple bounds.

\begin{lemma}
\label{lem:compgrowth}
The number of compacted trees of size $n$ satisfies the following bounds:
\[
n! \leq \ct_n \leq \rt_n \leq \frac{1}{n+1} \binom{2n}{n} \; n!.
\]
\label{lem:upper_bound}
\end{lemma}

\begin{proof}
Let us first consider the lower bound: 
Consider the subclass of chains. These are trees where the left child is always an internal edge and the right child is a pointer, see Figure~\ref{fig:R0}. Let $a_n$ be the number of chains with $n$ internal nodes. The leaf is the only such object of size~$0$. Hence, we have $a_0=1$. A chain of size $n+1$ can be constructed from a chain of size~$n$ by appending a new root node with a pointer. The pointer has $n+1$ possible locations to point to. This implies, $a_{n+1} = (n+1) a_n$. We get the lower bound $a_n = n!$. 

\begin{figure}[htb]
	\centering
	\includegraphics[width=0.6\textwidth]{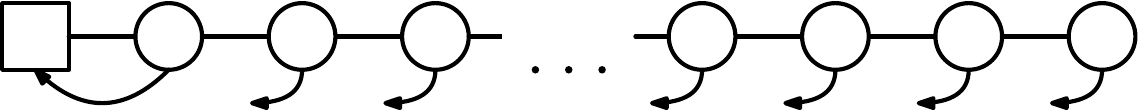}
	\caption{\small The number of compacted trees of size $n$ of right height at most $0$ is equal to $n!$.}
	\label{fig:R0}
\end{figure}

Let us now focus briefly on the upper bound: Consider all possible spines. There are $\Cat_n = \frac{1}{n+1}\binom{2n}{n}$ (Catalan numbers) such structures, as they are binary trees. 
Next, note that a binary tree of size $k$ has $k+1$ leaves. In our case these are pointers. By Proposition~\ref{prop:pointerrestriction} pointers can only point to previously discovered trees. Hence, every pointer has at most $k$ possibilities to point at. This proves the upper bound.
\end{proof}

The last result implies that the asymptotic growth of compacted trees satisfies
$
	\ct_n = \LandauO(n! \, 4^n \, n^{-3/2}),
$
but it is also bounded from below by $n!$.
This observation has two important implications. 
Firstly, an ordinary generating function for $(\ct_n)_{n\in\mathbb N}$ would have radius of
convergence equal to zero.
Hence, we will need to use exponential generating functions in order to ensure a non-zero radius of
convergence.
This idea will be used in the next sections. 
Secondly, combining it with our main result Theorem~\ref{theo:compactedmain} we directly get:

\begin{coro}
	\label{coro:mainexpgrowth}
	The exponential growth of compacted and relaxed binary trees of size $n$ is equal to $4^n$, i.e.
	\begin{align*}
		\lim_{n \to \infty} \frac{\log(c_n/n!)}{n} &= 4, & \text{ and } && 
		\lim_{n \to \infty} \frac{\log(r_n/n!)}{n} &= 4.
	\end{align*}
\end{coro}

\begin{proof}
	Observe that for any $k$ it holds that $c_{k,n} \leq c_n \leq r_n \leq n! \Cat_n$. Thus, the asymptotic form of the Catalan numbers and Theorem~\ref{theo:compactedmain} show the claim.
\end{proof}

In the next section we return to compacted binary trees of bounded right height and start to capture their nature with exponential generating functions.

\section{Operations on trees}
\label{sec:operations}

We have seen in the previous sections that the numbers $\cgf_n$ and $\rgf_n$ are growing at least like $n!4^n$.
Therefore we introduce exponential generating functions in order to get a non-zero radius of convergence.
But then there arises a problem in the construction: exponential generating functions are designed
for labeled objects, but we are dealing with unlabeled ones. Thus, we first investigate how the nature
of exponential generating functions reflects the construction of such enriched trees. 

The use of non-standard generating functions in the enumeration of DAGs is not new.
Robinson~\cite{robinson1973dags} introduced the so-called ``special generating function''
\begin{align*}
	A(t) &= \sum_{n \geq 0} a_n 2^{-\binom{n}{2}} \frac{t^n}{n!}
\end{align*}
to derive nice expressions of such generating functions for labeled DAGs.
This \emph{ad hoc} generating function seems not applicable in our context,
but exponential generating functions are. 

For this purpose, we restrict ourselves to a subclass: relaxed trees of bounded right height, and
we are going to derive their exponential generating functions. 
In this context we introduce the following notations: Let $\Rc$ be a combinatorial class.
Its exponential generating function is given by $\Rgf(z) = \sum_{n \geq 0} \rgf_n \frac{z^n}{n!}$
where $\rgf_n$ denotes the number of elements in $\Rc$ of size $n$.

\begin{lemma}(Adding a new root)
	\label{lem:approot}
	Let $\Rc$ be a combinatorial subclass of relaxed trees, and let $\Sc$
	be the combinatorial class whose elements consist of a new root node, with an element of $\Rc$
	as its left child, and with a pointer as its right child. 
	Then,
	\begin{align*}
		S(z) = z\Rgf(z).
	\end{align*}
\end{lemma}

\begin{proof}
	Consider a relaxed tree of $\Rc$ of size $n$.
	Adding a new root node with the considered tree as its left child creates a tree of size $n+1$. 
	The new pointer has $n+1$ possibilities, in particular it may point to one
	of the $n$ internal nodes or the leaf. On the level of generating functions this implies
	\begin{align*}
		S(z) &= \sum_{n \geq 0} (n+1) \rgf_n \frac{z^{n+1}}{(n+1)!} = z \Rgf(z). \qedhere
	\end{align*}
\end{proof}
	
With the help of this lemma, we are able to construct the generating function
of relaxed trees of right height equal to $0$.
Let $\Rc_0$ be the respective combinatorial class,
and $\Rgf_0(z) = \sum_{n \geq 0} \rgf_{0,n} \frac{z^n}{n!}$ be the associated generating function. 

\begin{coro}
	\label{coro:R0}
	The generating function of relaxed trees of right height equal to $0$ is
	\begin{align*}
		\Rgf_0(z) &= \frac{1}{1-z}, \qquad 
		\text{ and } \qquad
		\rgf_{0,n} = n!.
	\end{align*}
\end{coro}

\begin{proof}
	Such a tree is either just a leaf of size $0$ or it is constructed
	from an element of $\Rc_0$ by appending a new root node. Obviously,
	this construction does not increase the right height, and it constructs
	all such trees. On the level of generating functions this directly translates into
	\begin{align*}
		\Rgf_0(z) &= 1 + z \Rgf_0(z).
	\end{align*}
	Solving the equation and extracting coefficients gives the result. 
\end{proof}

This gives an alternative proof of the lower bound in Lemma~\ref{lem:compgrowth}. It nicely exemplifies how exponential generating functions model operations on compacted trees.

We proceed now with other operations on combinatorial classes and generating functions. The next
two might seem ``strange'' at first glance, as they do not produce relaxed trees. However, they
are the basic operations for the construction of other ones.

\begin{lemma}[Adding/deleting the root while ignoring pointers]
	\label{lem:adddelnopointer}
	Let $\Rc$ be a class of relaxed trees. Let $\Ic$ be the class of objects obtained from
	$\Rc$ by adding a new root node \emph{without pointer} (as its right child),
	and let $\Dc$ be the class obtained from $\Rc$ by deleting the root node but (if existent)
\emph{keeping its pointer}.\footnote{This means in particular, that a single leaf, being root of a
size 0 object, simply disappears. Furthermore, an object with a root having no pointers will
become disconnected at the root. The pointers from the right to the left subtree remain. However,
this construction will only be used when the root has a pointer.} Then,
	\begin{align*}		
		I(z) &= \int \Rgf(z) \, dz, \\
		D(z) &= \frac{d}{dz} \Rgf(z).
	\end{align*}
\end{lemma}

\begin{proof}
	Adding a new root node increases the size by one, whereas deleting it decreases it by one. Hence, 
	elements of $\Rc$ of size $n$ are in bijection with elements of
	$\Ic$ of size $n+1$ as well as with elements of $\Dc$ of size $n-1$, compare Figure~\ref{fig:adddel}. Therefore, we get
	\begin{align*}		
		I(z) &= \sum_{n \geq 0} \rgf_{n} \frac{z^{n+1}}{(n+1)!} = \int \Rgf(z) \, dz, \\
		D(z) &= \sum_{n \geq 1} \rgf_{n} \frac{z^{n-1}}{(n-1)!} = \frac{d}{dz} \Rgf(z). \qedhere
	\end{align*}
\end{proof}

\begin{figure}[htb]
	\centering
	\includegraphics[width=0.25\textwidth]{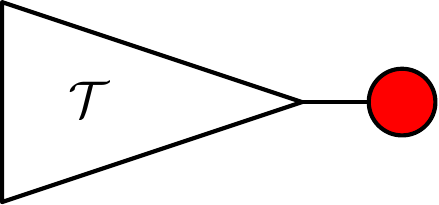}
	\qquad
	\includegraphics[width=0.25\textwidth]{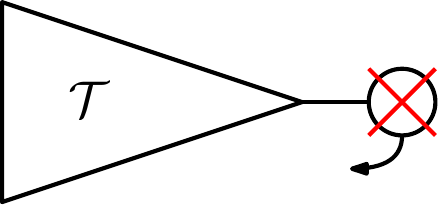}
	\qquad
	\includegraphics[width=0.25\textwidth]{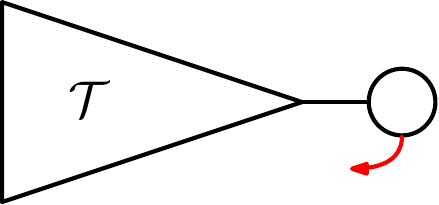}
	\caption{\small Adding a new root node without pointer, deleting a root
		node while preserving its (possible) pointer, 
		and adding a new pointer to the existing root node.}
	\label{fig:adddel}
\end{figure}

These constructions can then be used to derive the following two operations:

\begin{prop}[Sequences and pointers]
	\label{prop:seqpoint}
	The generating function $S(z)$ corresponding to the class obtained by 
appending an arbitrary (possibly empty but finite) sequence of nodes to the root (each with one pointer) to a class $\Rc$ is given by 
	\begin{align*}
		S(z) &= \frac{1}{1-z} R(z).
	\end{align*}
	The generating function $P(z)$ of the class obtained by adding a new,
	additional pointer to the root nodes of the objects of a class $\Rc$ is given by
	\begin{align*}
		P(z) &= z \frac{d}{dz} R(z) + \rgf_0.
	\end{align*}
\end{prop}

\begin{proof}
	This is a direct consequence of the Lemmas~\ref{lem:approot}
	and \ref{lem:adddelnopointer}, compare Figures~\ref{fig:adddel} and \ref{fig:addsequence}.
\end{proof}

\begin{remark}
Note that when applying several consecutive sequence constructions as defined above, then the
resulting structure looks like a single sequence construction. But we would get
several factors $1/(1-z)$ in the generating function, though. This would only be correct if we set a
marker after each application in order to remember where a sequence ends and the next one starts. 

Alternatively, we may simply forbid consecutive sequence constructions. In particular, this means
that $\Rc$ must be built in such a way that appending a sequence of nodes does not generate
consecutive sequence constructions. 

But all this is only a \emph{caveat} in the usage of the sequence construction. When building
relaxed and compacted trees, we never face consecutive sequence constructions, so there is no need
to pay attention to it in our context. 
\end{remark}

\begin{figure}[htb]
	\centering
	\includegraphics[width=0.8\textwidth]{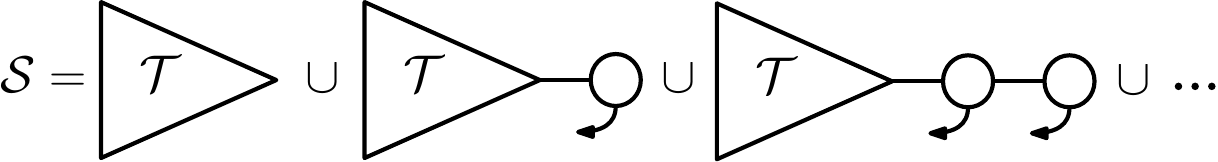}
	\caption{\small Appending a finite (possibly empty) sequence to the root node.}
	\label{fig:addsequence}
\end{figure}

Now we have all operations needed to continue our investigation of trees with bounded right height. 
In the next sections we show how this calculus is used to derive
differential equations for relaxed and compacted trees of bounded right height.

In the sequel, it will prove convenient to work with operators on generating functions.
For this purpose, we will use the same letters for the operators as were used for
the combinatorial classes (or generating functions).

\section{Relaxed binary trees}
\label{sec:relaxed}

We will now show how to use the calculus developed in Section~\ref{sec:operations}
to derive ordinary differential equations for the exponential generating
functions of relaxed trees of bounded right height. 
In this context we introduce the following notation: Let $\Rc$ be the combinatorial class
of relaxed trees. Its exponential generating function is given by
$\Rgf(z) = \sum_{n \geq 0} \rgf_n \frac{z^n}{n!}$ where $\rgf_n$
denotes the number of elements in $\Rc$ of size $n$. 
We denote the class of relaxed trees of right height at most~$k$ by $\Rc_k$
and its corresponding exponential generating function by $\Rgf_k(z) = \sum_{n \geq 0} \rgf_{k,n} \frac{z^n}{n!}$. 

We have derived $\Rgf_0(z)$ in Corollary~\ref{coro:R0} as
\begin{align*}
	\Rgf_0(z) &= \frac{1}{1-z} = \sum_{n \geq 0} n! \frac{z^n}{n!}.
\end{align*}

Let us now consider relaxed trees of right height \emph{at most one}.

\subsection{Relaxed trees of right height at most 1}
\label{sec:R1}

Let $\Rc_1$ be the combinatorial class of relaxed trees with right height at most $1$, compare Figure~\ref{fig:R1l}.
The corresponding generating function is given by $\Rgf_1(z) = \sum_{n \geq 0} \rgf_{1,n} \frac{z^n}{n!}$.

\begin{figure}[htb]
	\centering
	\includegraphics[width=0.9\textwidth]{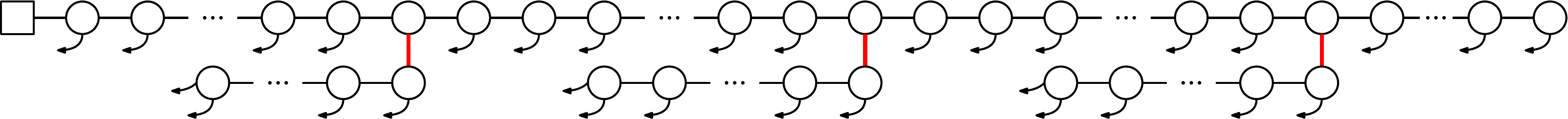}
	\caption{\small A relaxed tree from $\Rc_1$, i.e.~with right height at most $1$.}
	\label{fig:R1l}
\end{figure}

We will break the problem into smaller parts by decomposing $\Rc_1$ according to the following equation
\begin{align}
	\Rgf_1(z) &= \sum_{\ell \geq 0} \Rgf_{1,\ell}(z), \label{eq:R1decomp}
\end{align}
where $\Rgf_{1,\ell}(z)$ is the exponential generating function of
relaxed binary trees with exactly $\ell$ right subtrees, i.e.~$\ell$ 
right edges in the spine going from level $0$ to level $1$. 
Obviously, we have $R_{1,0}(z) = R_0(z) = \frac{1}{1-z}$. 
In order to get $R_{1,1}(z)$, we apply the previously developed constructions.
An illustration of such a tree is shown in Figure~\ref{fig:R11}.

\begin{figure}[htb]
	\centering
	\includegraphics[width=0.6\textwidth]{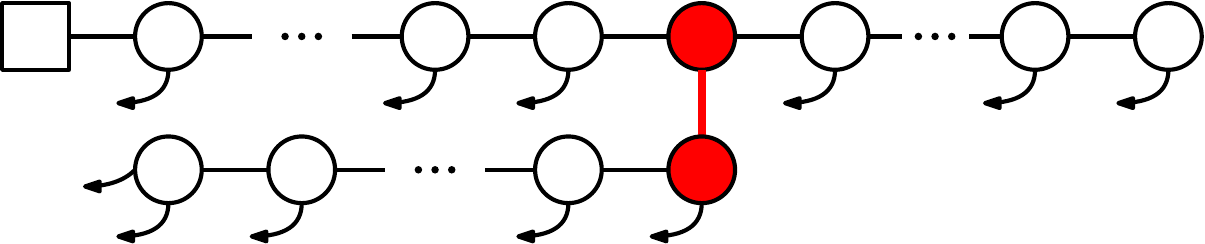}
	\caption{\small A relaxed tree with exactly one right edge in the spine.}
	\label{fig:R11}
\end{figure}

\begin{prop}
	\label{prop:R11}
	The generating function of relaxed trees with exactly one right edge in the spine is given by
	\begin{align*}
		\Rgf_{1,1}(z) &= 
			\frac{1}{1-z} \int{ \frac{1}{1-z} z \left( z \Rgf_{1,0}(z)\right)' \, dz} = 
			\frac{z^2}{2(1-z)^3}.
	\end{align*}
\end{prop}

\begin{proof}
	The idea is to decompose the structure of $\Rc_{1,1}$ into smaller parts
	which are in bijection to constructible classes. 
	\begin{enumerate}
		\item \label{item:R11a} 
		On level $0$ there is a unique node with one right edge, see Figure~\ref{fig:R11}.
Before this node there is a possibly empty sequence of nodes corresponding to the sequence
construction given by the operator $\Sc$. Call this the initial sequence. First consider a relaxed
tree with empty initial sequence, see Figure~\ref{fig:R11dropseq}.
		
		\begin{figure}[htb]
			\centering
			\includegraphics[width=0.3\textwidth]{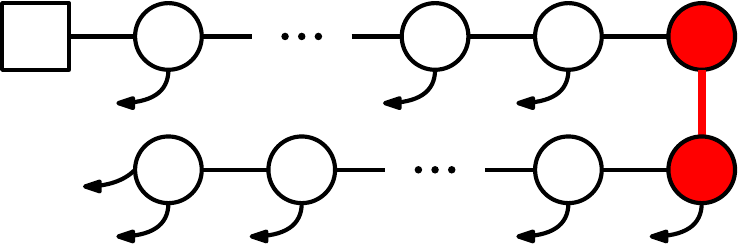}
			\caption{\small Step~$1$: An element of $\Rc_{1,1}$ with empty sequence of
initial nodes on level~$0$.}
			\label{fig:R11dropseq}
		\end{figure}
	
		\item \label{item:R11b} 
		On level $0$, the left child of the unique node with two children (and without pointer)
		is followed by a sequence of nodes, whose pointers may only point to nodes of
the sequence.
		This is an element of $\Rc_0$ and thus counted by $\Rgf_0(z)$.
		
		Furthermore, we see that the elements on level $1$ form a sequence with a cherry
		as its last element. Its pointers may also point to nodes from the sequence
discussed in the previous paragraph, which is in bijection with $\Rc_0$.
		By moving the $\Rc_0$-instance of level $0$ to the end of the sequence on level $1$
		we get a sequence containing one special node which has two pointers. 
		Then we delete the last node on level $0$, 
		compare with Figure~\ref{fig:R11decompb}. 
		
		In terms of generating functions we get
	\begin{align}
		\hat{R}_{1,0}(z) := \frac{1}{1-z} \underbrace{z\left( z \Rgf_{1,0}(z)
\right)'}_{\text{add root with two pointers}}. \label{eq:addrootwith2pointers}
	\end{align}
	Note that due to the cherry every element has at least one internal node. 	
		\begin{figure}[htb]
			\centering
			\includegraphics[width=0.35\textwidth]{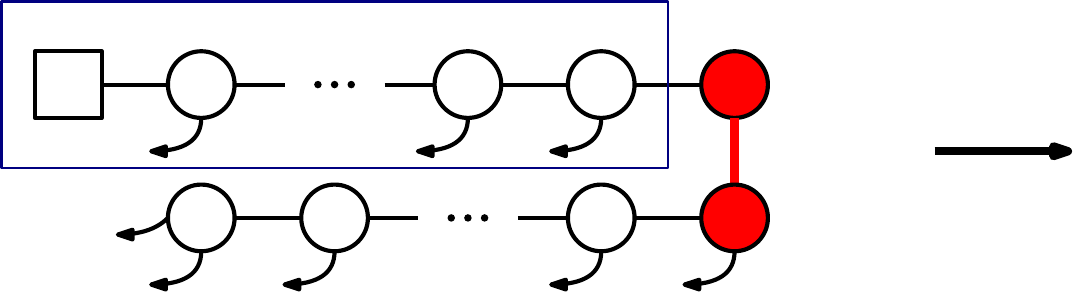}
			\qquad
			\includegraphics[width=0.5\textwidth]{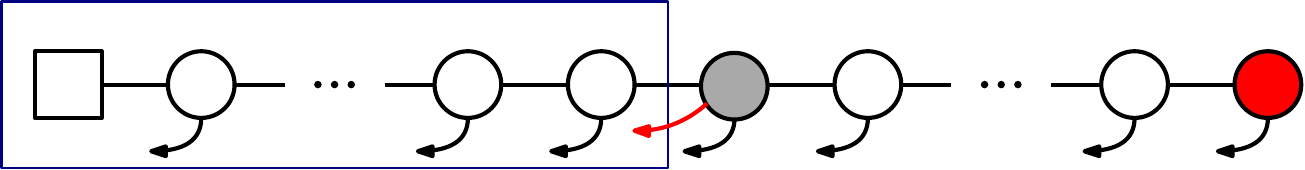}
			\caption{\small Step~$2$: Nodes of level~$0$ can only point to nodes on level~$0$ (left); 
				moving these nodes to level $1$ and deleting the remaining node at level~$0$
				gives $\hat{R}_{1,0}(z)$ (right).}
			\label{fig:R11decompb}
		\end{figure}

		\item
		Furthermore, notice that the node on level $0$ containing a right child (and not a right pointer)
		has no pointers. However, elements of the initial sequence may point to it.
		Therefore, we reinsert this node by adding it as a new root without pointer.
		The constructed object bijectively corresponds to the elements of $\Rc_{1,1}$ with
empty initial sequence. 
	
		\item 		
		Finally, we append an initial sequence (cf. Step~$1$). 
	\end{enumerate}
After those steps, the resulting object looks like shown in Figure~\ref{fig:R11decompfinished}:
a sequence with two special nodes, one having no pointer, the other one having two pointers. The class
of all such elements is in bijection with $\Rc_{1,1}$, as all the steps above can be reverted. 
		\begin{figure}[htb]
			\centering
			\includegraphics[width=0.7\textwidth]{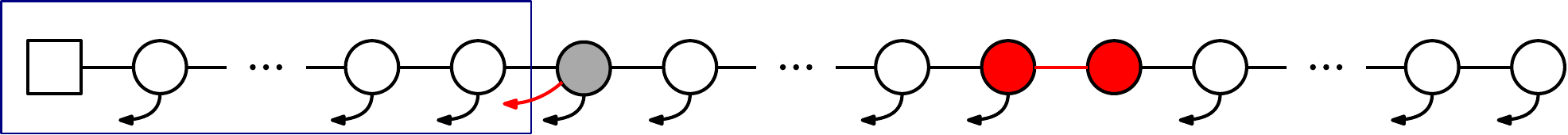}
			\caption{\small Step~$4$: The final sequence-like object bijectively corresponding to $\Rc_{1,1}$.}
			\label{fig:R11decompfinished}
		\end{figure}

Now we have to translate the operations performed in the four steps into algebraic operations on
generating functions. As already mentioned, after Step~$2$ the class of objects we get in that
way has generating function $\frac{1}{1-z} z\left( z \Rgf_{1,0}(z)\right)'$. The operation in Step~$3$ corresponds to integrating the generating function by Lemma~\ref{lem:adddelnopointer}. The final
step is the application of the operator $S$ of Proposition~\ref{prop:seqpoint} and therefore
generates a factor $\frac{1}{1-z}$, which completes the proof. 
\end{proof}

The main idea of the previous proof was to cut and glue the $\Rc_{1,1}$-instance in such a way
that a sequence-like object appears such that the process forms a bijection from $\Rc_{1,1}$ to
the class of sequence-like objects of the form shown in Figure~\ref{fig:R11decompfinished}. 
This new object has the advantage of being
constructible by the operations introduced in Section~\ref{sec:operations}. 

The previous decomposition captures all necessary mechanics to compute $\Rgf_{1,\ell}(z)$.

\begin{coro}
	\label{coro:R1l}
	The generating function of relaxed trees $\Rc_{1,\ell}$ with exactly $\ell$
	right edges in the spine from level $0$ to level $1$ is given by
	\begin{align*}
		\Rgf_{1,\ell}(z) &= \frac{1}{1-z} \int{ \frac{1}{1-z} z \left( z
\Rgf_{1,\ell-1}(z)\right)' \, dz}, \qquad \ell \ge 1,\\
		\Rgf_{1,0}(z) &= \Rgf_0(z) = \frac{1}{1-z}.
	\end{align*}
\end{coro}

\begin{proof}
	By cutting at the first right edge from level~$0$ to level~$1$,
	we observe a decomposition into an initial sequence, a right edge
	from level~$0$ to level~$1$ with its two endnodes being a sequence on level~$1$
	and an instance counted by $R_{1,\ell-1}(z)$. 
	The decomposition is exhibited in Figure~\ref{fig:R1lstructure}.
	Thus, we may reuse the construction from Proposition~\ref{prop:R11}
	by replacing the initial value $R_{1,0}(z)$ by $R_{1,\ell-1}(z)$.
\end{proof}

\begin{figure}[htb]
	\centering
	\includegraphics[width=0.8\textwidth]{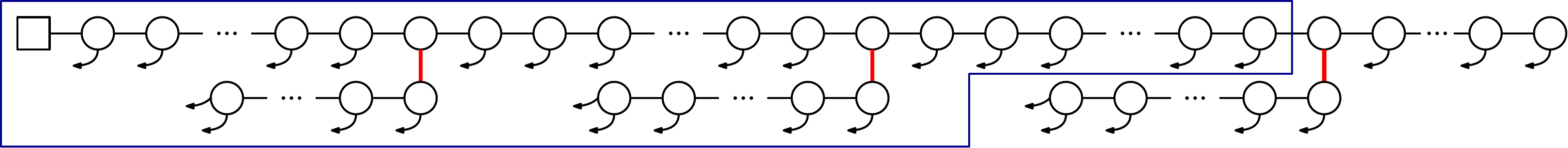}
	\caption{\small A recursive decomposition of elements from $\Rc_{1,\ell}$.
		}
	\label{fig:R1lstructure}
\end{figure}
Finally, we are able to combine the previous results to derive
the generating function of~$\Rc_1$. We need the classical notation of \emph{double factorials}:
\begin{align*}
	n!! &:= \prod\limits_{k=0}^{\lfloor \frac{n-1}{2} \rfloor} (n - 2k), \qquad \text{ for }  n \in \N.
\end{align*}
\begin{theo}
	\label{theo:R1}
	The exponential generating function of relaxed trees of right height at most $1$ is D-finite and satisfies
	\begin{align*}
		(1-2z) \Rgf_1'(z) - \Rgf_1(z) &= 0, &
		\Rgf_1(0)=1.
	\end{align*}
	The closed-form formula and the coefficients are given by
	\begin{align*}
		\Rgf_1(z) &= \frac{1}{\sqrt{1-2z}}, &
		\rgf_{1,n} &= (2n-1)!!.
	\end{align*}
\end{theo}
\begin{remark}
For more on the general background of $D$-finite functions we refer to Stanley's excellent book~\cite{stan99}.
Furthermore, note that the falling double factorials count many combinatorial families, see \OEIS{A001147}.
A bijective interpretation of this behavior was found by the last author in~\cite{Wallner19R1}.
\end{remark}
\begin{proof}
	We start with the result of Corollary~\ref{coro:R1l}. But instead of the integral representation, we use the following differential equation valid for $\ell \geq 1$:
	\begin{align*}
		(1-z) \left( (1-z) \Rgf_{1,\ell}(z) \right)' =  z\left( z \Rgf_{1,\ell-1}(z) \right)'.
	\end{align*}
	Remembering the initial decomposition \eqref{eq:R1decomp} and summing over all $\ell \geq
1$ we get
	\begin{align*}
		(1-z) \left( (1-z) \left( \Rgf_1(z) - \Rgf_{1,0}(z) \right) \right)' =  z\left( z \Rgf_{1}(z) \right)'.
	\end{align*}
	Rearranging this equation and replacing $\Rgf_{1,0}(z)$ by $\Rgf_0(z)$ we get
	\begin{align}
		(1-2z) \Rgf_1'(z) - \Rgf_1(z) - (1-z) \left((1-z)\Rgf_0(z)\right)' = 0. \label{eq:R1dfinite}
	\end{align}
	Now, $\Rgf_0(z) = \frac{1}{1-z}$, hence the differential equation simplifies to 
	\begin{align*}
		(1-2z) \Rgf_1'(z) - \Rgf_1(z) = 0.
	\end{align*}
	Solving this equation by separation of variables yields the closed-form expression.
	Finally, the extraction of the coefficients is easy using
	$\frac{1}{\sqrt{1-4z}} = \sum_{n \geq 0} \binom{2n}{n} z^n$.
\end{proof}

\subsection{Relaxed trees of right height at most 2}
\label{sec:R2}

Let $\Rc_2$ be the combinatorial class of relaxed trees with right height at most $2$, compare Figure~\ref{fig:R2}.
The corresponding generating function is given by $\Rgf_2(z) = \sum_{n \geq 0} \rgf_{2,n} \frac{z^n}{n!}$.

\begin{figure}[htb]
	\centering
	\includegraphics[width=0.9\textwidth]{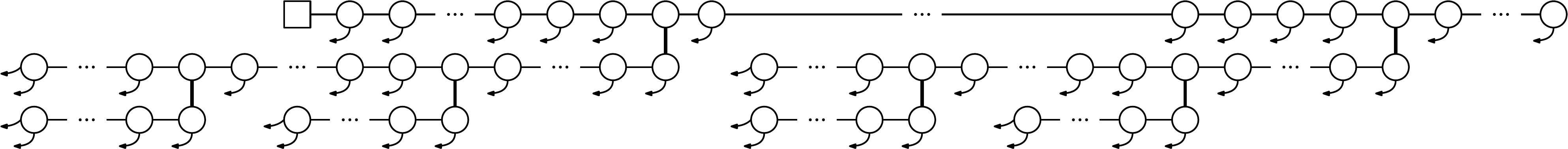}
	\caption{\small A relaxed tree from $\Rc_2$, i.e.~with right height at most $2$.}
	\label{fig:R2}
\end{figure}

In the same fashion as before, we will break the problem
into smaller parts by decomposing $\Rc_2$ into
\begin{align}
	\Rgf_2(z) &= \sum_{\ell \geq 0} \Rgf_{2,\ell}(z), \label{eq:R2decomp}
\end{align}
where $\Rgf_{2,\ell}(z)$ is the exponential generating function of relaxed trees of right height
at most~$2$ with exactly $\ell$ right edges in the spine
going from level $0$ to level $1$. Obviously, we have $R_{2,0}(z) = R_0(z) = \frac{1}{1-z}$. 

\begin{remark}\label{rem:incond}
Note that, as seen in the sequel, the functions $R_{k,0}(z)=\frac{1}{1-z}$ are in fact the
perturbation of the recurrence of differential equations we are currently building. Moreover, they
also uniquely determine the initial condition of this recurrence. Therefore, we will sloppily call
these functions as well as others in the same role ``initial conditions''. This should not be
confused with the initial conditions of the differential equations themselves. Those do not play
any role in our arguments, so the risk of confusion should be low. 
\end{remark}

\begin{prop}
	\label{prop:R21}
	The exponential generating function of relaxed trees of right height at most~$2$
	with exactly one right edge from level $0$ to level $1$ in the spine satisfies
	\begin{align}
		(1-2z) \left( (1-z) \Rgf_{2,1}(z) \right)'' 
			- \left( (1-z) \Rgf_{2,1}(z) \right)' 
			- \left( z ( z\Rgf_{2,0}(z) )' \right)' = 0. \label{eq:R21dfinite}
	\end{align}
\end{prop}

\begin{proof}
	The main idea is to decompose an object of $\Rc_{2,1}$ again into $4$ parts
	(compare with Figure~\ref{fig:R21decomp}): an initial sequence, the first right edge
	from level $0$ to level $1$, the sequence on level $0$ after this right edge,
	and an instance of $\Rc_{1}$ starting on level $1$ after this right edge.
	Then we use the same transformation idea as in the proof of Proposition~\ref{prop:R11}.
	We take the sequence on level $0$ after the right edge and move it to the
	end of the $\Rc_1$-instance. Note that this is legitimate concerning the pointers.
	But it generates a node with two pointers within an instance of $\Rc_1$.
	With respect to this $\Rc_1$-instance this change happens on its top level to the very 
	left. 
	
	We can now delete the initial sequence and the level $0$ node of the right edge,
	as they can be created again by known operations.
	Let $\mathcal F$ be the class of objects obtained after performing the above operations and $F(z)$ be its generating function.
	Schematically, this class is shown in the bottom of Figure~\ref{fig:R21decomp}.
	By Lemma~\ref{lem:adddelnopointer} and Proposition~\ref{prop:seqpoint} we get
	\begin{align*}
		F(z) &= ((1-z) \Rgf_{2,1}(z))'.
	\end{align*}
	
	Note that $F(z)$ is associated to structures with right height at most $1$.
	It is nearly an instance of $\Rc_1$. There are only two differences: 
	
	First, it has a special construction after its last right edge.
	With respect to the differential equation~\eqref{eq:R1dfinite}, which corresponds to 
	the class $\Rc_1$, this changes the initial condition $\Rc_0$ (recall Remark~\ref{rem:incond}!). 
	Due to linearity, we can reuse this specification by replacing the initial condition.	
	On the level of generating functions this corresponds to replacing $R_0(z)$ by $\frac{1}{1-z} (z \Rgf_{2,0}(z))'$, because a (possibly empty) sequence 
	is followed by a node with a double pointer and an instance of $R_{2,0}(z)$, which is in this case another sequence (compare with Figure~\ref{fig:R11decompb}). 
	Let $\mathcal G$ be the corresponding combinatorial class and $G(z)$ its generating function.
	By~\eqref{eq:R1dfinite} we have
	\begin{align*}
		(1 &-2z)G'(z) - G(z) - (1-z) ((1-z)G_0(z))' = 0, \quad \text{ with } \quad
		G_0(z) = \frac{1}{1-z} (z \Rgf_{2,0}(z))'.
	\end{align*}	

	Second, due to the unique right edge from level $0$ to level $1$, every object in $\mathcal F$ 
	has at least one particular node, namely the red node on level $1$ (compare the transformation shown in Figure~\ref{fig:R21decomp}).
	Let us describe the unfavourable case we need to avoid, namely that there is no such node.
	Looking back at the beginning of the transformation, this case is equivalent to the fact
that the subtree on levels $1$ and $2$ is empty, or in other words, the right edge going from level
0 to level 1 (red egde in Figure~\ref{fig:R21decomp}) is only a pointer. 
	During the transformation process the $\Rc_0$-instance at the end of level 0 is moved to
level one, where it forms then an $\Rc_0$-instance with an additional pointer, namely the above-mentioned pointer being the red edge in Figure~\ref{fig:R21decomp} in the unfavourable case. The
generating function of such structures is $(z \Rgf_{2,0}(z))'$, as we start with an
$\Rc_0$-instance (which can equivalently be regarded as an $\Rc_{2,0}$-instance), add a new root
and then delete this new root, but keep its pointer. Hence, in order to correct for the
unfavourable case we need to subtract $(1-z)G_0(z)$. We get 
	\begin{align*}
		F(z) &= G(z) - (1-z) G_0(z).
	\end{align*}
	This yields 
	\begin{align}
		G(z) &= \left( (1-z) \Rgf_{2,1}(z) \right)' + (z \Rgf_{2,0}(z))'. \label{eq:Gexpressed}
	\end{align}
	Finally, putting everything together some straightforward calculations show \eqref{eq:R21dfinite}.	
\end{proof}

\begin{figure}[htb]
	\centering
	\includegraphics[width=0.6\textwidth]{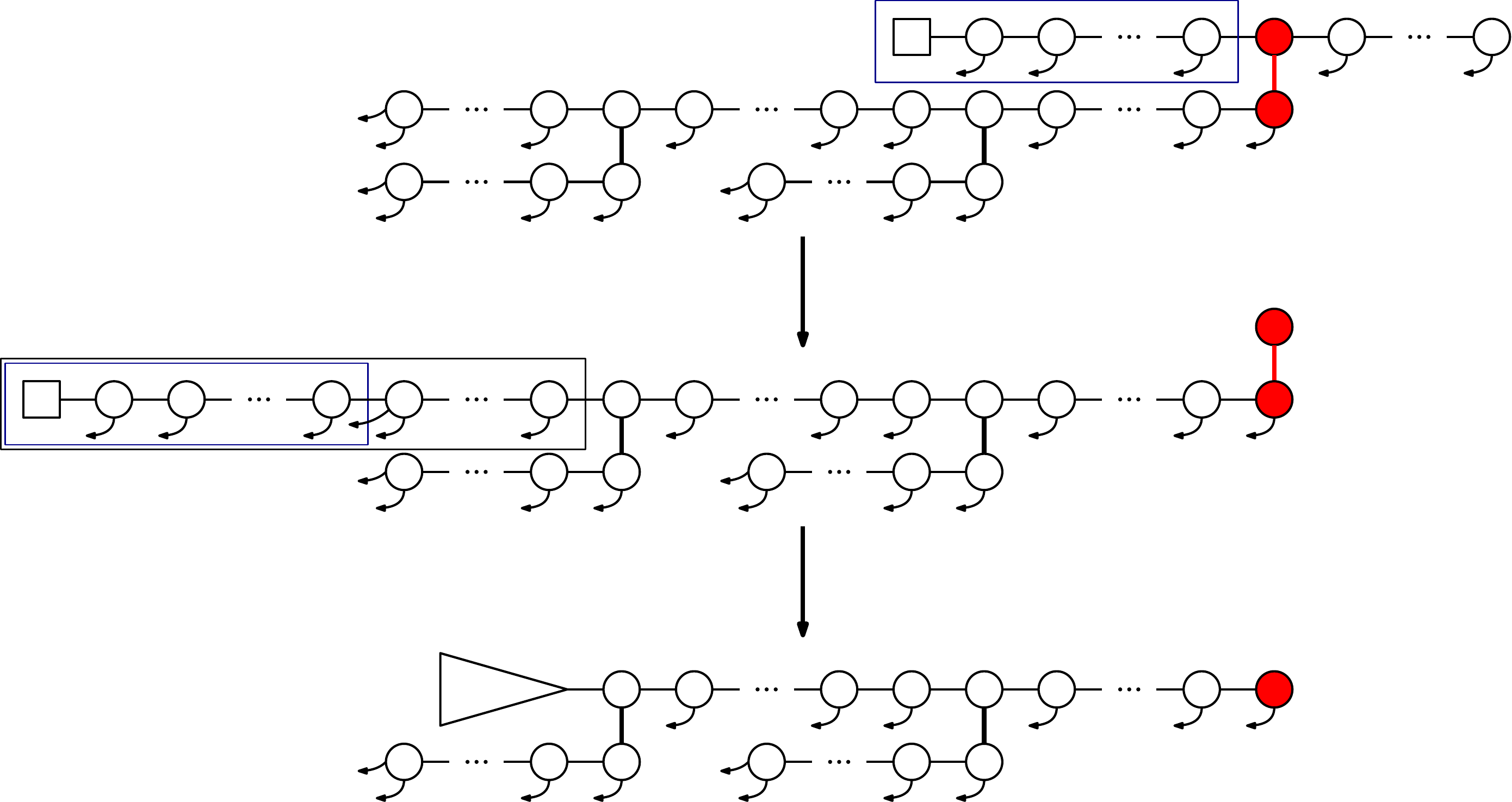}
	\caption{\small Transforming a structure of $\Rc_{2,1}$ into an instance of $\Rc_{1}$.}
	\label{fig:R21decomp}
\end{figure}

As in the $\Rgf_{1,\ell}(z)$ case, we get $\Rgf_{2,\ell}(z)$ for $\ell \geq 2$
by a recursive application of the previous arguments.

\begin{coro}
	\label{coro:R2l}
	The generating function of relaxed trees with right height at most $2$, 
	and exactly $\ell$ right edges in the spine from level $0$ to level $1$ is given by
	\begin{align*}
		&(1-2z) \left( (1-z) \Rgf_{2,\ell}(z) \right)'' 
			- \left( (1-z) \Rgf_{2,\ell}(z) \right)' 
			- \left( z ( z\Rgf_{2,\ell-1}(z) )' \right)' = 0, \qquad \ell \ge 1,\\
		&\Rgf_{2,0}(z) = \Rgf_0(z) = \frac{1}{1-z}.
	\end{align*}
\end{coro}

\begin{proof}
	By cutting at the first right edge from level $0$ to level $1$,
	we observe a decomposition into an initial sequence, a right edge
	from level $0$ to level $1$ with $2$ nodes, a sequence on level~$1$
	and an instance counted by $R_{2,\ell-1}(z)$. 
	Thus, we may reuse the construction from the proof of Proposition~\ref{prop:R21}
	by replacing the initial value $R_{2,0}(z)$ with $R_{2,\ell-1}(z)$.
\end{proof}
Note that for the final result it is crucial that we found homogeneous differential equations.

\begin{theo}
	\label{theo:R2}
	The exponential generating function of relaxed trees of right height at most $2$
	is D-finite and satisfies
	\begin{align*}
		(z^2-3z+1) \Rgf_2''(z) + (2z-3) \Rgf_2'(z) &= 0, &
		\Rgf_2(0)=1,~\Rgf_2'(0)=1.
	\end{align*}
	A closed-form formula and the coefficients are given by
	\begin{align*}
		\Rgf_2(z) &= -\frac{2}{\sqrt{5}} \artanh \left( \frac{2z-3}{\sqrt{5}} \right) 
							- \frac{1}{\sqrt{5}}\left( \log \left( \frac{7 + 3\sqrt{5}}{2} \right) - \pi i \right), \\
		\rgf_{2,n} &= \frac{(n-1)!}{\sqrt{5}} \left( \left( \frac{1 + \sqrt{5}}{2} \right)^{2n} - \left( \frac{1 - \sqrt{5}}{2} \right)^{2n} \right).
	\end{align*}
\end{theo}

\begin{proof}
	Again, let us take the result of Corollary~\ref{coro:R2l} and
	sum over all $\ell \geq 1$, while remembering the decomposition
	\eqref{eq:R2decomp}. By linearity this gives
	\begin{align}
		(1-2z) \left( (1-z) (\Rgf_2(z) - \Rgf_{2,0}(z) ) \right)'' 
			- \left( (1-z) (\Rgf_2(z) - \Rgf_{2,0}(z) ) \right)' 
			- \left( z ( z\Rgf_2(z) )' \right)' = 0. \label{eq:R2subst}
	\end{align}
	A simplification gives
	\begin{align*}
		(z^2-3z+1) \Rgf_2''(z) 
			+ (2z-3) \Rgf_2'(z) 
			- (1-2z)((1-z)\Rgf_{2,0}(z))'' + ((1-z) \Rgf_{2,0}(z) )' = 0.
	\end{align*}
	Inserting the initial value $\Rgf_{2,0}(z) = \frac{1}{1-z}$ we get the D-finite expression. 
	The correctness of the closed-form formula can then be easily checked
	with a computer algebra system.
	
	In order to extract the coefficients of $R_2(z)$ we observe that the differential
	equation can be simplified further by an integration with respect to $z$.
	Thus, it is equivalent to
	\begin{align*}
		(z^2-3z+1) \Rgf_2'(z) &=1, & 
		\Rgf_2(0) = 1,
	\end{align*}
	as $\Rgf_2'(0)=1$. Next, observe that as we are dealing with exponential
	generating functions, the derivative is just a shift on the level of coefficients.
	In other words, $[z^n] R_2(z) = [z^{n-1}] R_2'(z)$.
	Therefore, a partial fraction decomposition enables a direct extraction of the coefficients. 
\end{proof}

\subsection{Relaxed trees of right height at most \texorpdfstring{$k$}{k}}
\label{sec:relaxedheightk}

The approach from the previous section can be generalized to an arbitrary bound $k\ge 2$ for the 
right height. Let $\Rgf_k(z) = \sum_{n \geq 0} \rgf_{k,n} \frac{z^n}{n!}$ be the corresponding
generating function. The idea is to use the previous construction, and to derive a differential equation for $\Rgf_k(z)$ from the one of $\Rgf_{k-1}(z)$.

\begin{figure}[htb]
	\centering
	\includegraphics[width=0.8\textwidth]{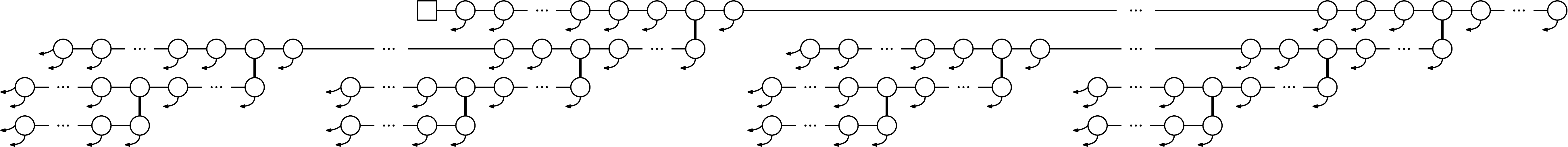}
	\caption{\small A relaxed tree from $\Rc_3$, i.e.~with right height at most $3$.}
	\label{fig:R3}
\end{figure}

We introduce a family of linear differential operators $L_k$, $k \geq 1$, which describe the
differential equations constructed for $\Rgf_k(z)$. Let $D$ denote the differential operator
$\frac{d}{dz}$ and $1$ the identity operator, i.e. $1 \cdot F(z)=F(z)$. For example, $(D \cdot z) (F(z)) = \frac{d}{dz} \left( z F(z) \right)$.
We want to stress at this point that the operators are in general not commutative.
\begin{theo}[Differential operators]
	\label{theo:diffoprelaxed}
	Let $(L_k)_{k \geq 0}$ be a family of differential operators given by
	\begin{align*}
		L_0 &= (1-z) \cdot 1,\\
		L_1 &= (1-2z) D - 1, \\
		L_{k} &= L_{k-1} \cdot D - L_{k-2} \cdot D^2 z, \qquad k \geq 2. 
	\end{align*}
	Then the exponential generating function $\Rgf_k(z)$ of relaxed binary trees with right height at most $k$ satisfies for $k \geq 1$
	\begin{align}\label{mainODE}
		L_k \cdot \Rgf_k(z) &= 0.
	\end{align}
\end{theo}

\begin{proof}
	For $k \geq 1$ we derive two families of operators:
	The differential operator $L_k$ and an auxiliary operator $H_k$ for the inhomogeneity such that
	\begin{align*}
		L_k \cdot \Rgf_k(z) &= H_k \cdot \Rgf_0(z).
	\end{align*}
	For $k=1$ we derived in~\eqref{eq:R1dfinite} the claimed form $L_1 \cdot R_1(z) = H_1
\cdot R_0(z)$ with $H_1 = L_0 D (1-z)$. 
	
	We continue with the case $k=2$. The explicit differential operator is given in Theorem~\ref{theo:R2}. We will now show how the operator can be constructed from the ones for $R_1(z)$ and $R_0(z)$ in the language of operators.
	
	In Proposition~\ref{prop:R21} we have derived the necessary
	substitution to get the differential equation of $\Rgf_2(z)$
	from the one of $\Rgf_1(z)$. The idea was to decompose $\Rc_2$ with respect to the number $\ell$
	of right edges from level $0$ to level $1$, see Figure~\ref{fig:R21decomp}. 
	This transformation creates an $\Rc_1$-like structure with a new initial condition $G_{0,\ell}(z)$ and the constraint not to be empty. 
	
	From~\eqref{eq:R1dfinite} we get the generic differential equation for $\Rc_1$-like structures with generating function $G_{\ell}(z)$ as
	\begin{align*}
		L_1 \cdot G_{\ell}(z) = H_1 \cdot G_{0,\ell}(z). 
	\end{align*}
	
	First, the new initial condition is given by  
	\begin{align*}
		G_{0,\ell}(z) &= \left(\frac{1}{1-z} D z\right) \cdot \Rgf_{2,\ell-1}(z).
	\end{align*}
	Second, the $\Rc_1$-like class being in bijection to $\Rc_{2,\ell}$ cannot be empty, and the initial sequence on level $0$ has to be appended. Thus, the substitution~\eqref{eq:Gexpressed} has to be used where $\Rgf_{2,1}(z)$ is replaced by $\Rgf_{2,\ell}(z)$, and $\Rgf_{2,0}(z)$ by $\Rgf_{2,\ell-1}(z)$. This gives for $\ell \geq 1$
	\begin{align*}
		L_1 \left( (D (1-z)) \cdot  \Rgf_{2,\ell}(z) + (D z) \cdot \Rgf_{2,\ell-1}(z)
\right) - H_1\left( \left(\frac{1}{1-z} D z\right) \cdot \Rgf_{2,\ell-1}(z) \right) = 0.
	\end{align*}
	Summing over $\ell \geq 1$ and recalling that $\Rgf_2(z) = \sum_{\ell \geq 0} \Rgf_{2,\ell}(z)$ we get
	\begin{align*}
		L_1 \cdot D \cdot \Rgf_2(z) - L_0 \cdot D^2 \cdot z \cdot \Rgf_2(z) = L_1 \cdot D \cdot (1-z) \cdot \Rgf_{2,0}(z).
	\end{align*}
	On the left we see the differential operator $L_2$ applied to $\Rgf_2(z)$ and on the right the inhomogeneity operator $H_2$ applied to $\Rgf_{2,0}(z)$. 
	Inserting $\Rgf_{2,0}(z) = \frac{1}{1-z}$ shows the claim for $k=2$. 
	
	Finally, for larger $k$, we can recycle the previous arguments for $k=2$
	and apply them recursively. This holds, as we may again cut an instance of $\Rc_k$
	 at the first right edge in the spine from level $0$ to level $1$ and decompose
	it in the repeatedly shown fashion, compare with Figure~\ref{fig:R21decomp}. 
	Then the same reasoning as in Section~\ref{sec:R2} allows us to extract
	the differential equation of $\Rgf_k(z)$ from the one of $\Rgf_{k-1}(z)$ by
	\begin{align}
		L_{k} &= L_{k-1} \cdot D - H_{k-1} \cdot \frac{1}{1-z} \cdot D \cdot z, &
		H_{k} &= L_{k-1} \cdot D \cdot (1-z). \label{eq:relLHk}
	\end{align}	
	Hence, by induction the claim holds.
\end{proof}

Let us apply the last theorem and compute the first few differential equations. 
\begin{equation*}
	(1-2z) \frac{d}{dz} \Rgf_1(z) - \Rgf_1(z) = 0,
\end{equation*}
\begin{equation*}
	(z^2 - 3z + 1) \frac{d^{2}}{dz^{2}} \Rgf_2(z) + (2z-3) \frac{d}{dz} \Rgf_2(z) = 0,
\end{equation*}
\begin{equation*}
	(3z^2-4z+1) \frac{d^{3}}{dz^{3}} \Rgf_3(z) + (9z-6) \frac{d^{2}}{dz^{2}}\Rgf_3(z) + 2 \frac{d}{dz} \Rgf_3(z) = 0,
\end{equation*}
\begin{equation*}
	-(z^3-6z^2+5z-1) \frac{d^{4}}{dz^{4}} \Rgf_4(z) - (6z^2-24z+10) \frac{d^{3}}{dz^{3}} \Rgf_4(z) - (6z-11) \frac{d^{2}}{dz^{2}} \Rgf_4(z) = 0.
\end{equation*}
The initial conditions of the differential equations can be obtained successively from lower order solutions. 
In particular, note that due to the construction the coefficients of $z^0,z^1,\ldots,z^{k+1}$ of $\Rgf_k(z)$
are the first $k+2$ elements of the counting sequence of relaxed trees, as a tree of size $k+1$
has always right height at most $k$. Thus with $\Rgf_k(z)$ we can enumerate all relaxed trees up to size $k+1$. 

Next, we take a closer look at these operators. 

\begin{theo}[Properties of $L_k$]
	\label{theo:Dkprop}
	For any $k \in \N$, let $L_k$ be as in Theorem~\ref{theo:diffoprelaxed}. Let $\ell_{k,i}(z) \in \C[z]$ be such that 
	\begin{align}\label{operater_L_k}
		L_k &= \ell_{k,k}(z) D^k
						+ \ell_{k,k-1}(z) D^{k-1}
						+ \ldots
						+ \ell_{k,0}(z).
	\end{align}
	Then we have
	\begin{align*}
		\ell_{k,0}(z) &= 0, \\
		\ell_{k,1}(z) &= \ell_{k-1,0}(z) - 2 \ell_{k-2,0}(z),\\
		\ell_{k,i}(z) &= \ell_{k-1,i-1}(z) - (i+1) \ell_{k-2,i-1}(z) - z \ell_{k-2,i-2}(z), \qquad 2 \leq i \leq k-1,\\
		\ell_{k,k}(z) &= \ell_{k-1,k-1}(z) - z \ell_{k-2,k-2}(z).
	\end{align*}
	The initial polynomials are $\ell_{0,0}(z) = 1-z$, $\ell_{1,0}(z) = -1$, and $\ell_{1,1}(z) = 1-2z$.
\end{theo}

\begin{proof}
	The initial polynomials are given by Theorem~\ref{theo:diffoprelaxed}. The shape
\eqref{operater_L_k} of the operator $L_k$ follows by induction using its recursive definition. 
Using an ansatz and comparing coefficients gives the recurrence relations for $\ell_{k,i}(z)$.
\end{proof}

The asymptotic behavior (according to $n$)
of the number $\rgf_{k,n}$ of relaxed trees with right height at most $k$
is governed by these differential equations.
These differential equations belong to a known class~\cite[Chapter VII.9]{flaj09}. Consider an ordinary generating function of the kind
\begin{align}
	\label{eq:difeqmerogen}
	D^r Y(z) + a_1(z) D^{r-1} Y(z) + \cdots + a_r(z) Y(z) = 0,
\end{align}
where the $a_i \equiv a_i(z)$ are meromorphic in a simply connected domain $\Omega$. Given a
meromorphic function $f(z)$, let $\omega_{\zeta}(f)$ be the order of the pole of $f$ at $\zeta$, and $\omega_\zeta(f)=0$ meaning that $f(z)$ is analytic at $\zeta$. 

\begin{definition}[Regular singularity, {\cite[p.~519]{flaj09}}]
	\label{def:regularsing}
	The differential equation~\eqref{eq:difeqmerogen} is said to have a singularity at $\zeta$ if at least one of the $\omega_\zeta(a_i)$ is positive. The point $\zeta$ is said to be a \emph{regular singularity} if 
	\begin{align*}
		\omega_\zeta(a_1) \leq 1, \qquad \omega_\zeta(a_2) \leq 2, \qquad \ldots, \qquad \omega_\zeta(a_r) \leq r,
	\end{align*}
	and an \emph{irregular singularity} otherwise.
\end{definition}

\begin{definition}[Indicial polynomial, {\cite[p.~520]{flaj09}}]
	\label{def:indicialpoly}
	Given an equation of the form~\eqref{eq:difeqmerogen} and a regular singular point $\zeta$, the \emph{indicial polynomial} $I(\alpha)$ at $\zeta$ is defined as
	\begin{align*}
		I(\alpha) &= \alpha^{\underline{r}} + \delta_1 \alpha^{\underline{r-1}} + \cdots + \delta_r, & \text{ where } &&
		\alpha^{\underline{\ell}} &:= \alpha(\alpha-1)\cdots(\alpha-\ell+1)
	\end{align*}
	and $\delta_i := \lim_{z \to \zeta} (z-\zeta)^i a_i(z)$. The \emph{indicial equation at $\zeta$} is the algebraic equation $I(\alpha) = 0$.
\end{definition}

The following technical lemma will be needed to derive the asymptotics for the solutions of the special type of differential equations given in Theorem~\ref{theo:odeasymptregsing}. 

\def\ff#1#2{#1^{\underline{#2}}}
\let\set\mathbb

\begin{lemma}
\label{lem:structuresimplefactor}
Let $p_0,\dots,p_r\in\set C[x]$ and consider the differential operator
\[
 L = p_r D^r + \cdots + p_1 D + p_0.
\]
Suppose that $x$ is a simple factor of~$p_r$, and suppose
that for some $\alpha\in\set C$, a solution of $L f(x) = 0$ admits a generalized
series solution $f(x)=\sum_{n\in\alpha+\set Z} c_n x^n$.
Then the coefficient sequence $(c_n)_{n\in\alpha+\set Z}$ satisfies a recurrence of the form 
\begin{center}
\begin{tabular}{rll}
$\Bigl(([x^1]p_r)(n-r+1)+  ([x^0]p_{r-1})\Bigr)$ & $\ff{n}{r-1} \; c_{n}$ & \\
  $+ \quad[\cdots]$ & $\ff{(n-1)}{r-2} \; c_{n-1}$ & \\
  $+ \quad[\cdots]$ & $\ff{(n-2)}{r-3} \; c_{n-2}$ & \\
  $+ \quad \dots~$ & & \\
  $+ \quad[\cdots]$ & $c_{n-s}$ & $= \quad 0,$
\end{tabular}
\end{center}
where $[\cdots]$ are certain polynomials in~$n$ and $s$ is some fixed nonnegative integer. 
\end{lemma}

\begin{proof}
We have $x^j D^i f = \sum_{n\in\alpha+\set Z} c_n \ff ni x^{n-i+j} = \sum_{n\in\alpha+\set Z} c_{n+i-j} \ff{(n+i-j)}i x^n$
for all $i, j\in\set N$.\\
Write $p_i = \sum_j p_{i,j}x^j$ for $i=0,\dots,r$, in the understanding that $j$ runs through all integers,
but $p_{i,j}$ is zero for all negative and almost all positive indices~$j$. By assumption, we know
that $p_{r,0}=0\neq p_{r,1}$. 

It follows that $p_i D^i f = \sum_{n\in\alpha+\set Z} \sum_j p_{i,j} c_{n+i-j} \ff{(n+i-j)}i x^n$ for $i=0,\dots,r$. Then
\[
 L f = \sum_{n\in\alpha+\set Z} \sum_{i=0}^r\sum_j p_{i,j} c_{n+i-j} \ff{(n+i-j)}i x^n = 0
\]
implies, by comparing the coefficients of~$x^n$,
\begin{equation}\label{eq:de2re}
0 = \sum_{i=0}^r\sum_j p_{i,j} c_{n+i-j} \ff{(n+i-j)}i
= \sum_j \sum_{i=0}^r p_{i,i+j}\ff{(n-j)}i c_{n-j} 
\end{equation}
for all $n\in\alpha+\set Z$. 

Consider a fixed $j\in\set Z$. From the definition $\ff{(n-j)}i=(n-j)(n-j-1)\cdots(n-j-i+1)$ it follows
that $\ff{(n-j)}i\mid\ff{(n-j)}{i+1}$ for every $i\in\set N$.
Therefore, if $k$ is minimal such that $p_{k,k+j}\neq0$, then $\ff{(n-j)}k\mid \sum_{i=0}^r p_{i,i+j}\ff{(n-j)}i$.

Note also that for each fixed $j$, the polynomial $\sum_{i=0}^r p_{i,i+j}\ff{(n-j)}i$ is non-zero if and only if
at least one of the coefficients $p_{i,i+j}$ are non-zero, because the falling factorials form a basis of the vector space
of polynomials.

For $j<-r$, we have $i+j<0$ for all $i=0,\dots,r$, and therefore $p_{i,i+j}=0$ for all $i$ and $\sum_{i=0}^r p_{i,i+j}\ff{(n-j)}i=0$.
Therefore there are no terms $c_{n-j}$ with $j<-r$ present in equation~\eqref{eq:de2re}.

For $j=-r$, we have $i+j<0$ for all $i=0,\dots,r-1$, and therefore $p_{i,i+j}=0$ for all these~$i$.
In addition, we have $p_{r,r-r}=p_{r,0}=0$ by assumption, so again $\sum_{i=0}^r p_{i,i+j}\ff{(n-j)}i=0$,
and no term $c_{n-j}$ with $j=-r$ is present in equation~\eqref{eq:de2re}.

Next, for $j=-r+1$ we have $p_{r,r+(-r+1)}=p_{r,1}\neq0$ by assumption, so the term $c_{n-(r-1)}$ \emph{does}
occur in equation~\eqref{eq:de2re}.
Moreover, since $p_{i,i+(-r+1)}=0$ for all $i<r-1$, we have $\sum_{i=0}^r p_{i,i+j}\ff{(n-j)}i=p_{r,1}\ff{(n-j)}r +
p_{r-1,0}\ff{(n-j)}{r-1}=(p_{r,1}n + p_{r-1,0})\ff{(n+r-1)}{r-1}$.

In general, for any $j>-r+1$, we have $p_{i,i+j}=0$ for all $i<-j$ and therefore $\ff{(n-j)}{-j}\mid\sum_{i=0}^r p_{i,i+j}\ff{(n-j)}i$.
(The understanding here is that $\ff{(n-j)}{-j}=1$ if $-j$ is not positive.)
Substituting $n-r+1$ for~$n$, we have shown the stated form of the recurrence. 
\end{proof}

If $\zeta$ is a regular singularity of a differential equation, then all solutions of the differential equation behave for $z\to\zeta$ like $(z-\zeta)^\alpha\log(z-\zeta)^\beta$ for some $\alpha\in\C,\beta\in\N$. 
The exponents $\alpha$ are roots of the indicial polynomial, and the exponents of the logarithmic terms are related to multiple roots of the indicial polynomial and roots at integer distances. More precisely, in our case the following theorem will be applicable. It is a variant of \cite[Theorem~VII.9]{flaj09} which works due to $\omega_{\zeta}(a_i) = 1$ for all $i=1,\ldots,r$.

\begin{theo}
	\label{theo:odeasymptregsing}
	Consider the differential equation~\eqref{eq:difeqmerogen} and a
	regular singular point $\zeta$ such that $\omega_{\zeta}(a_i) \le 1$
	for all $i = 1,\ldots,r$, and $\delta_1 := \lim_{z \to \zeta} (z-\zeta) a_1(z) \geq 0$. 
	Then, the vector space of all solutions defined in a slit neighborhood
	of $\zeta$ has a basis of $r$ functions, where $r-1$ functions are of the form
	\begin{align*}
		&(z - \zeta)^{m} H_m(z-\zeta), & m&=0,1,\ldots,r-2,
	\end{align*}
	with functions $H_m$ being analytic at $0$ and satisfying $H_m(0) \neq 0$. 
	The $r$-th basis function depends on $\delta_1$: 
	\begin{enumerate}
		\item For $\delta_1 \in \{0,1,\ldots,r-1\}$ it is of the form
		\begin{align*}
			(z - \zeta)^{r-1-\delta_1} H(z-\zeta) \log(z-\zeta);
		\end{align*}
		\item For $\delta_1 \in \{r,r+1,\ldots\}$ it is of the form
		\begin{align*}
			(z - \zeta)^{r-1-\delta_1} H(z-\zeta) + H_0(z-\zeta)\left(\log(z-\zeta)\right)^k, \quad \text{ with } \quad k \in \{0,1\};
		\end{align*}
		\item For $\delta_1 \not\in \Z$ it is of the form
		\begin{align*}
			(z - \zeta)^{r-1-\delta_1} H(z-\zeta);
		\end{align*}
	\end{enumerate}
		where $H$ is analytic at $0$, with $H(0) \neq 0$.
\end{theo}

\begin{proof}
	Due to $\omega_{\zeta}(a_i) \le 1$ we get by the definition of the
	indicial polynomial that $\delta_i=0$ for $i \geq 2$. Hence, it is given by
	\begin{align*}
		I(\alpha) &= \alpha^{\underline{r}} + \delta_1 \alpha^{\underline{r-1}} = \alpha^{\underline{r-1}} (\alpha-r+1+\delta_1).
	\end{align*}
	Therefore, the roots are $0,1,\ldots,r-2$, and $r-1-\delta_1$.
	
	Let us treat the consecutive range of roots $0,1,\ldots,r-2$ first. 
	Consider the equivalent recurrence relation for the coefficients
	$(c_n)_{n\in \N}$ of the series solution expanded at $\zeta$. It has the form
	\begin{align*}
		I(n) y_n = \Phi(y_{n-1},\ldots,y_{n-N}),
	\end{align*} 
	where $I(n)$ is the indicial polynomial, $N=\max_i(r-i+\deg(a_i(z)))$,
	and $\Phi$ is a linear operator with polynomial coefficients in $n$.
	Let $\alpha$ be a root of the indicial polynomial, and consider the sequence
	$(c_n)_{n\in \N}$ extended to $\Z$ with $c_n = 0$ for $n<\alpha$ for $\alpha = 0,1,\ldots,r-2$.
	At $n=\alpha$ we have
	\begin{align}
		\label{eq:birthcoeff}
		0 \cdot y_{\alpha} = \Phi(0,\ldots,0).
	\end{align} 
	Hence, $y_{\alpha}$ can be chosen arbitrarily. By Lemma~\ref{lem:structuresimplefactor},
for each choice the recurrence uniquely extends the sequence towards $+\infty$. Therefore, each
root $\alpha$ gives rise to a different solution of our recurrence relation. The set of all these
solutions is linearly independent.
	The consecutive range of zeros implies that the  values $y_{0},\ldots,y_{r-2}$ can be chosen arbitrarily, as they do not interfere with each other. Such a situation does not give rise to any logarithmic terms.
	
	Next, let us treat the remaining basis solution associated to $r-1-\delta_1$.
	
	In the first case, there is a multiple root of order $2$. Then, the classical theory of linear differential equations implies the appearance of logarithmic terms, see \cite{henr77,waso87,ince44,SchlesingerI1968}.

	In the second case, the situation is analogous to~\eqref{eq:birthcoeff}: The solution starts to exist at $n=r-1-\delta_1$. But this solution then needs to be continued further, and at $n=0$, we might have a problem. Then, there could emerge a logarithmic term or not.
	This depends on the specific problem.
	If the solution cannot be extended, then a logarithmic factor multiplied
	with the solution at $n=0$ is added, see \cite{ince44}. 
	
	In the third case, the root does not interfere with the other solutions,
	as the difference with any other root is not an integer.
	Thus, it can be continued without problems, and has the claimed form.
\end{proof}

By Theorem~\ref{theo:diffoprelaxed} the differential equations associated to relaxed trees
are of the kind~\eqref{eq:difeqmerogen}. 
The roots of the leading term are under these conditions
responsible for the singularities of the solutions. 
The dominant one is as usual the one closest to the origin.
Our first aim is to show that for every bounded right height
there exists a unique dominant singularity.
For this purpose we start with the analysis of the polynomials $\ell_{k,i}(z)$.
They are strongly connected to a famous family of polynomials:
the \emph{Chebyshev polynomials}, see, e.g.,~\cite[Chapter~18]{NIST:DLMF} or \cite[Chapter~22]{AbramowitzStegun1964}.

\begin{definition}[Chebyshev polynomials]
	The \emph{Chebyshev polynomials of the first kind} $T_n(z)$ are defined by the recurrence relation
	\begin{align*}
		T_0(z) &= 1,\\
		T_1(z) &= z, \\
		T_{n+2}(z) &= 2z T_{n+1}(z) - T_{n}(z).
	\end{align*}
	The \emph{Chebyshev polynomials of the second kind} $U_n(z)$ are defined by the recurrence relation
	\begin{align*}
		U_0(z) &= 1,\\
		U_1(z) &= 2z, \\
		U_{n+2}(z) &= 2z U_{n+1}(z) - U_{n}(z).
	\end{align*}
\end{definition}

\begin{lemma}[Transformed leading coefficient]
	\label{lem:chebyshev}
	Let $\ell_{k,i}(z)$ be the coefficients of the operator $L_k$ from Theorem~\ref{theo:Dkprop}.
	Then, for the leading coefficient we get
	\begin{align*}
		\ell_{k,k}(z) &= z^{\frac{k+2}{2}} U_{k+2}\left(\frac{1}{2 \sqrt{z}}\right) 
		               = \sum_{n = 0}^{\lfloor \frac{k+2}{2} \rfloor} (-1)^n \binom{k+2-n}{n} z^n,
	\end{align*}
	where $U_k(z)$ are the Chebyshev polynomials of the second kind.
\end{lemma}

\begin{proof}
		We start with the recurrence relation of $\ell_{k,k}(z)$ from Theorem~\ref{theo:Dkprop}.
		Replacing $z$ by $\frac{1}{4z^2}$ and multiplying by $(2z)^{k+2}$, we get
		\begin{align*}
			(2z)^{k+2} \ell_{k,k}\left(\frac{1}{4z^2}\right) = 
				2z \cdot (2z)^{k+1} \ell_{k-1,k-1}\left(\frac{1}{4z^2}\right)
				- (2z)^{k} \ell_{k-2,k-2}\left(\frac{1}{4z^2}\right),
		\end{align*}
		and we recognize the recurrence relation for the Chebyshev polynomials of the second kind for $(2z)^{k} \ell_{k-2,k-2}\left(\frac{1}{4z^2}\right)$,  see~\cite[Section~18.9]{NIST:DLMF}.
		Transforming the initial conditions, gives $U_2(z)$ and $U_3(z)$ respectively. 
		
		The closed form is derived from the well-known formula
		$U_k(z) = \sum_{n = 0}^{\lfloor \frac{k}{2} \rfloor} (-1)^n \binom{k-n}{n}
(2z)^{k-2n}$.
\end{proof}

We will also need the following result on $\ell_{k,k-1}(z)$. Its structure is directly related to the one of $\ell_{k,k}(z)$.

\begin{lemma}[Transformed $\ell_{k,k-1}(z)$]
	\label{lem:lkk1}
	For the coefficient $\ell_{k,k-1}(z)$ of the operator $L_k$ from Theorem~\ref{theo:Dkprop},
	we get
	\begin{align*}
		\ell_{k,k-1}(z) &= \frac{k}{2} \ell'_{k,k}(z), 
	\end{align*}
	for $k \geq 1$.
\end{lemma}

\begin{proof}
  By Theorem~\ref{theo:Dkprop} the claim holds for $k =1$ and $k=2$. We proceed by induction. Assume the claim holds for $1 \leq i \leq k$. Then, differentiating both sides of the defining equation of $\ell_{k,k}(z)$ given in Theorem~\ref{theo:Dkprop} we get
	\begin{align*}
		\ell'_{k,k}(z) &= \ell'_{k-1,k-1}(z) - z \ell'_{k-2,k-2}(z) - \ell_{k-2,k-2}(z). 	
	\end{align*}
	Next, we apply the induction hypothesis and get
	\begin{align*}
		\ell'_{k,k}(z) &= \frac{2}{k-1} \ell_{k-1,k-2}(z) - z \frac{2}{k-2} \ell_{k-2,k-3}(z) - \ell_{k-2,k-2}(z).	
	\end{align*}
	Finally, by rearranging the equation and utilizing the defining recurrence relation for $\ell_{k,k-1}(z)$ we prove (omitting the arguments)
	\begin{align*}
							\ell'_{k,k}	&= \frac{2}{k} \Big(\underbrace{ \ell_{k-1,k-2} - z \ell_{k-2,k-3} - k \ell_{k-2,k-2} }_{= \ell_{k,k-1}} \Big) 
								+ \frac{1}{k} \Big(\underbrace{ \ell_{k-1,k-1} - 2z \ell'_{k-2,k-2} + k \ell_{k-2,k-2} }_{= 0 }\Big),
	\end{align*}
	where the last expression is equal to $0$, as we know the polynomial $\ell_{k,k}(z)$ explicitly from Lemma~\ref{lem:chebyshev}.
\end{proof}

Chebyshev polynomials are well-studied objects.
We summarize some important results (for our analysis) in the following lemma.
\begin{lemma}
	\label{lem:lkkproperties}
	The roots of $\ell_{k,k}(z)$ are real, positive, and distinct.
	Let $\rho_k$ be the smallest real root of $\ell_{k,k}(z)$. Then, $R_k(z)$ is singular at $\rho_k$ and we have
	\begin{align*}
		\rho_k &= \frac{1}{4\cos^2\left(\frac{\pi}{k+3}\right)}.
	\end{align*}
	Furthermore, $\rho_k$ is not a root of $\ell_{k,k-1}(z)$.
\end{lemma}

\begin{proof}
	The results follow from the well-known results on Chebyshev polynomials~\cite[Section~18.5]{NIST:DLMF}.
	In particular, the roots $x_{k,j}$ of $U_k(z)$ admit the closed-form expressions
	\begin{align*}
		x_{k,j} &= \cos\left(\frac{j \pi}{k+1}\right).
	\end{align*}
	This implies the closed-form expression of $\rho_k$.
	The last claim follows from the closed-form expression of $\ell_{k,k-1}(z)$ from Lemma~\ref{lem:lkk1} and the fact that the roots of Chebyshev polynomials are all simple.
	
	Finally, note that $\rho_{k} \leq \rho_{k-1}$ and $\rho_0=1$ is the singularity of $R_0(z)=\frac{1}{1-z}$. Let $\mu_k$ be the dominant singularity of $R_k(z)$. We prove by induction that $\mu_k = \rho_k$. Combinatorially, it is clear that $\mu_k \leq \mu_{k-1} = \rho_{k-1}$. Furthermore $\mu_k$ must be related to $x_{k,j}$. Thus, as the the roots of the Chebyshev polynomials are interlacing we can only have $\mu_k = \rho_k$. 
\end{proof}

Note that $\rho_0 = 1$, $\rho_1 = \frac{1}{2}$, and $\rho_2=\frac{2}{3+\sqrt{5}}$ are exactly the singularities of $R_0(z)$, $R_1(z)$, and $R_2(z)$, respectively. 
Furthermore, with this information, we are finally able to characterize the indicial polynomials.

\begin{prop}
	\label{prop:indicialpoly}
	The indicial polynomial $I_k(\alpha)$ at $\rho_k$ of the $k$-th differential equation is given by $I_k(\alpha) = \alpha^{\underline{k-1}} (\alpha-(\frac{k}{2}-1))$. 
\end{prop}

\begin{proof}
	By Definition~\ref{def:indicialpoly} we need to show that $\delta_i=0$ for $i \geq 2$ and
$\delta_1 = \frac{k}{2}$. The first claim holds by Lemma~\ref{lem:lkkproperties}, as $\ell_{k,k-i}(z)/\ell_{k,k}(z)$ has no higher-order poles for all $i \geq 1$.
	
	Let us reformulate the second claim:
	\begin{align}
		\label{eq:delta1explicit}
		\delta_1 &= \lim_{z \to \rho_k} \frac{\ell_{k,k-1}(z)}{\frac{\ell_{k,k}(z)}{z-\rho_k}} = \frac{\ell_{k,k-1}(\rho_k)}{\ell'_{k,k}(\rho_k)} = \frac{k}{2},
	\end{align}
	where the second equality sign holds because of de l'Hospital's rule and Lemma~\ref{lem:lkkproperties} ($\rho_k$ is not a root of $\ell_{k,k-1}(z)$). The last equality holds by Lemma~\ref{lem:lkk1} 
\end{proof}

With the help of the following lemma, we are able to simplify the indicial polynomials further. 
\begin{lemma}
	\label{lem:zeropolys}
	For $k \geq 2$ and $0 \leq i \leq \lfloor \frac{k-2}{2} \rfloor$ we have $\ell_{k,i}(z) \equiv 0$.
\end{lemma}

\begin{proof}
	Let us start with the cases $i=0$ and $i = 1$. 
	As defined in Theorem~\ref{theo:Dkprop} we have $\ell_{k,0}(z) = 0$ for $k\geq 2$.  The case $i =1$ is valid for $k \geq 4$. Then, we have
	\begin{align*}
		\ell_{k,1}(z) &= \ell_{k-1,0}(z) - \ell_{k-2,0}(z) = 0. 
	\end{align*}
	
	For the cases $i \geq 2$ we use induction on $k$. 
	Assume the claim holds for $2,\ldots,k-1$ and arbitrary $i$. Then, we have 
	\begin{align*}
		\ell_{k,i}(z) &= \ell_{k-1,i-1}(z) - (i+1)\ell_{k-2,i-1}(z) - z \ell_{k-2,i-2}(z) =0. 
	\end{align*}
	In all three cases it is easy to check that $i \leq \lfloor \frac{k-2}{2} \rfloor$ and the induction hypothesis implies that these terms are equal to $0$.
\end{proof}
Hence, the differential equation of order $k$ is actually a differential equation of order 
\begin{alignat*}{5}
	\tilde{r} &:= \left\lceil \frac{k}{2} \right\rceil & \quad \text{ for the function } \quad &
	\tilde{R}_k(z) &:= \frac{d^{\lfloor \frac{k}{2}\rfloor}}{dz^{\lfloor \frac{k}{2}\rfloor}} R_k(z).
\end{alignat*}
In other words, we have
\begin{align}
	\label{eq:relaxedreducedode}
	\ell_{k,k}(z) D^{\tilde{r}} \cdot \tilde{R}_k + 
	\ell_{k,k-1}(z) D^{\tilde{r}-1} \cdot \tilde{R}_k + \dots +
	\ell_{k,\lfloor \frac{k}{2}\rfloor}(z) \tilde{R}_k = 0.
\end{align}

\begin{coro}
	\label{coro:indicialpolyreduced}
	Let $\tilde{I}_k(\alpha)$ be the indicial polynomial at $\rho_k$ of the reduced
	differential equation~\eqref{eq:relaxedreducedode}. Then, 
	\begin{align*}
		\tilde{I}_k(\alpha) &=
			\begin{cases}
				\alpha^{\underline{\tilde{r}-1}} \left(\alpha+1\right), & \text{if } $k$ \text{ even},\\
				\alpha^{\underline{\tilde{r}-1}} \left(\alpha+\frac{1}{2}\right), & \text{if } $k$ \text{ odd}.
			\end{cases}
	\end{align*}
\end{coro}

\begin{proof}
	This is a direct consequence of Proposition~\ref{prop:indicialpoly}. 
	As only the order of the differential equation changed but not its coefficients, we get
	$$\tilde{I}_k(\alpha) = \alpha^{\underline{\tilde{r}}} + \delta_1 \alpha^{\underline{\tilde{r}-1}} 
	                      = \alpha^{\underline{\tilde{r}-1}} \left(\alpha - \left\lceil \frac{k}{2} \right\rceil + \frac{k}{2} +  1\right).               
	$$
	Considering the even and odd case separately yields the result.
\end{proof}

After these technical steps, we can finally prove our first main result.

\begin{proof}[Proof of Theorem~\ref{theo:relaxededmain}]
	By Lemma~\ref{lem:lkkproperties}, $\rho_k$ is the dominant singularity of $R_k(z)$. 
	Furthermore, the pole of $\frac{\ell_{k,i}(z)}{\ell_{k,k}(z)}$ at $z=\rho_k$ for $i=1,\ldots,k-1$
	is of order one for $i \geq 1$. Thus, by Definition~\ref{def:regularsing} it is a regular
singularity of the differential equation. 
	
	Furthermore, by Corollary~\ref{coro:indicialpolyreduced} the set of roots of the reduced
indicial polynomial is $\{-1,0,1,\ldots,\tilde{r}-2\}$ for even $k$ and $\{-\frac{1}{2},0,1,\ldots,\tilde{r}-2\}$ for odd $k$.
	In both cases by Theorem~\ref{theo:odeasymptregsing},
	a basis in a slit neighborhood of $\rho_k$ consists of the analytic functions
	\begin{align*}
		(1-z/\rho_k)^s H_s(1-z/\rho_k),
	\end{align*}
	for $s=0,\ldots,\tilde{r}-2$ where $H_s$ is analytic and nonzero at $0$, and a singular function
	\begin{align} \label{singular_fun}
		\begin{cases}
			\frac{1}{1-z/\rho_k} H(1-z/\rho_k) + G(1-z/\rho_k) \log(1-z/\rho_k), & \text{ if $k$ is even},\\
			\frac{1}{\sqrt{1-z/\rho_k}} H(1-z/\rho_k), & \text{ if $k$ is odd},\\
		\end{cases} 
	\end{align}
	with functions $G,H$ being analytic and nonzero at $0$. These functions form a basis of the solution space of~\eqref{eq:relaxedreducedode}.
	
	In order to obtain a basis of the solution space of the original differential equation
\eqref{operater_L_k}, we need to integrate $\left\lfloor \frac{k}{2} \right\rfloor$ times.
	The analytic basis functions remain analytic and the singular one singular.
	As there is always just one singular function, and we know that $\Rgf_k(z)$ is singular at $\rho_k$, this function must be responsible for the asymptotic growth.
	We get a singular expansion for $z \to \rho_k$ of the kind
	\begin{align*}
		R_k(z) &\sim 
				\begin{cases}
					\tilde\gamma_k (1-z/\rho_k)^{k/2-1}\log\left(\frac{1}{1-z/\rho_k}\right),  & \text{ if $k$ is even},\\
					\tilde\gamma_k (1-z/\rho_k)^{k/2-1},  & \text{ if $k$ is odd}.
				\end{cases}
	\end{align*}
Theorem~\ref{theo:odeasymptregsing} implies that $\tilde\gamma_k$ is a nonzero constant. As
$R_k(z)$ is the generating function of a counting sequence, the sign of $\tilde\gamma_k$
must be such that the coefficients of the asymptotic main term of $R_k(z)$ are eventually positive.
	Finally, applying the transfer theorems~\cite{flaj09}, the claim holds with 
\[
\gamma_k = \frac{\tilde\gamma_k}{\Gamma(-k/2+1)} \text{ for odd $k$ and } \gamma_k =
(-1)^{k/2+1} \Gamma\left(\frac k2\right)\tilde\gamma_k\text{ for even $k$}. \qedhere
\]
\end{proof}
	
	Let us comment on the even case. It is a priori not clear if this logarithmic term in
\eqref{singular_fun} appears or not (if not we set $G\equiv 0$). But due to the appearance of the term with the polar singularity, the logarithmic term does not influence the asymptotic main term. Obviously, it plays a role for the error terms. 
	For specific cases, we can of course answer this question. For $k=2$, we have seen in Section~\ref{sec:R2} that there are no logarithmic terms. However, in this case, the reduced indicial polynomial is only of order $1$, see Corollary~\ref{coro:indicialpolyreduced}. Therefore, the consecutive range of roots starting with $0$ does not exist.
	
	For $k=4$, logarithmic terms appear. In this case we have the operator 
	\begin{align*}
		(-z^3+6 z^2-5 z+1)D^2+(-6 z^2+24 z-10)D+(11-6 z)
	\end{align*}
	and the expansion point $\rho_4$ that is a root of $-z^3+6 z^2-5 z+1$.
	Then, the solution space is generated by the following two series:

\def\O{\LandauO}

\begin{tabular}{l l}
\label{eq:relaxedsols}
$1$ & \phantom{$+\tfrac{3255\rho_4^2-175989 \rho_4+805466}{49392} (\rho_4-z)^3 \log (z-\rho_4)$} \\
$-\tfrac{48 \rho_4^2-267 \rho_4+118}{14} (\rho_4-z)$ & \\
$+\tfrac{249 \rho_4^2-2340 \rho_4+5560}{588} (\rho_4-z)^2$ & \\
$-\tfrac{206442 \rho_4^2-1141941 \rho_4+502699}{7056} (\rho_4-z)^3$ & \\
$+\O((\rho_4-z)^{4}),$ & 
\end{tabular}

\medskip
and \\
\smallskip
 

\begin{tabular}{l l}
$(z-\rho_4)^{-1}$ & $+0\,(z-\rho_4)^{-1}\log(z-\rho_4)$ \\
$+0$ & $-\tfrac{6 \rho_4^2-33 \rho_4+14}{7} \log(z-\rho_4)$ \\
$+\tfrac{159 \rho_4^2-888 \rho_4+440}{196} (\rho_4-z)$ & $+\tfrac{3 \rho_4^2-54 \rho_4+194}{98} (\rho_4-z) \log (z-\rho_4)$ \\
$-\tfrac{2484 \rho_4^2-14439 \rho_4+9541}{1764} (\rho_4-z)^2$ & $-\tfrac{3834 \rho_4^2-21183 \rho_4+9221}{588} (\rho_4-z)^2 \log(z-\rho_4)$ \\
$+\tfrac{4199487 \rho_4^2-23629269 \rho_4+12255158}{592704} (\rho_4-z)^3$ & $+\tfrac{3255\rho_4^2-175989 \rho_4+805466}{49392} (\rho_4-z)^3 \log (z-\rho_4)$ \\
$+\O((\rho_4-z)^{4}\log(z-\rho_4)).$ & 
\end{tabular}

\section{Compacted binary trees}
\label{sec:compacted}

After the successful application of exponential generating functions
to relaxed trees of bounded right height, we will extend this method to compacted binary trees.
In this section we solve the problem of finding the generating function of compacted trees of bounded right height. We denote the class of compacted trees of right height at most $k$ by $\Cc_k$ and its corresponding exponential generating function by $\Cgf_k(z) = \sum_{n \geq 0} \cgf_{k,n} \frac{z^n}{n!}$. 

As every subtree in a relaxed tree of right height at most $0$ is unique, by Corollary~\ref{coro:R0} we immediately get 
\[
	\Cgf_0(z) = \frac{1}{1-z}.
\]

\subsection{The cherry operator}
\label{sec:C1}

We start with the subclass $\Cc_1$ of compacted trees of right height at most $1$.
The same ideas as in Section~\ref{sec:R1} are used in the analysis.
However, this case is more subtle as we have to guarantee uniqueness of the subtrees.
The main observation in this context is that in order to establish uniqueness
of the subtrees one needs to restrict the pointers of the cherries, see Proposition~\ref{prop:characterizingcompactedtrees}.

Consider a situation where the pointers of a cherry are pointing into a tree of size $k$.
Thus,  every pointer has $k+1$ possibilities ($+1$ due to the leaf).
In a relaxed setting this would mean that there are $(k+1)^2$ different configurations. 

In a compacted tree every internal node (or spine node) corresponds to a unique subtree.
Therefore, the cherry has only $(k+1)^2-k$ different options, see also Theorem~\ref{theo:comprecursion}. Let us introduce the corresponding operator now.

\begin{lemma}[Cherry operator]
	Let $\Cc$ be a class of compacted trees and $\Kc$ be the class obtained from $\Cc$
	by adding a new node with two pointers, where the decompacted tree of this new node
	(left pointer is left child, right pointer is right child) is not part of $\Cc$.  Then,
	\begin{align*}
		K(z) &= z \left(zC(z)\right)' - \int z C'(z) \, dz \\
			  &= z^2 C'(z) + \int C(z) \, dz.
	\end{align*}
\end{lemma}

\begin{proof}
	The first term corresponds to the (unconstrained) operation of adding a root with two pointers, see~\eqref{eq:addrootwith2pointers}. The second one is responsible for the correction, by deleting the number of subtrees which are already part of $\Cc$, see Figure~\ref{fig:cherryoperator}: 
	
	Consider a tree of $\Cc$ of size $k$. The integrand creates a pointer attached to the root possibly pointing to all elements of the subtree. The integration operator adds a new root node without a pointer. By attaching the newly created pointer to this new root, and changing the pointer in the case of it pointing to the leaf by letting it point to the old root, we generate $k$ new elements from this specific tree: A new root with a pointer to every internal node of the tree.
	This is exactly the number of elements which we need to subtract in order to ensure uniqueness.
	
	The second expression results from an integration by parts of the first one.
\end{proof}

\begin{figure}[htb]
	\centering
	\includegraphics[width=0.8\textwidth]{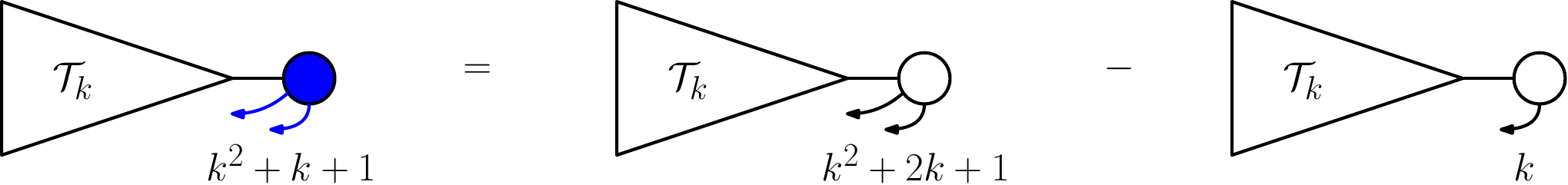}
	\caption{\small The construction of the cherry operator. The formulas below the pointers state the possible destinations of the pointers in the tree $\Tc_k$. The left tree is the desired one, the other ones are constructible ones. }
	\label{fig:cherryoperator}
\end{figure}

Let us also define the corresponding operator which performs the previous operation:
\begin{align*}
	K(C(z)) := z \left(zC(z)\right)' - \int z C'(z) \, dz.
\end{align*}

Next, we decompose $\Cc_1$ into
\begin{align*}
	\Cgf_1(z) &= \sum_{\ell \geq 0} \Cgf_{1,\ell}(z), 
\end{align*}
where $\Cgf_{1,\ell}(z)$ is the exponential generating function of compacted trees of right height at most $1$ with exactly $\ell$ right edges in the spine going from level $0$ to level $1$.  

\begin{coro}
	\label{coro:C1l}
	The generating function of compacted trees with exactly $\ell$ right edges from level $0$ to level $1$ in the spine is given by
	\begin{align*}
		\Cgf_{1,\ell}(z) &= \frac{1}{1-z} \int \frac{1}{1-z} K\left( C_{1,\ell-1}(z) \right) \, dz , \qquad \ell \geq 1,\\
		\Cgf_{1,0}(z) &= \frac{1}{1-z}.
	\end{align*}
\end{coro}

\begin{proof}
	The construction is analogous to the one of Corollary~\ref{coro:R1l}.
	The only difference is the use of the cherry operator in~\eqref{eq:addrootwith2pointers}.
\end{proof}	

\begin{theo}
	\label{theo:C1}
	The exponential generating function of compacted trees of right height at most~$1$ is D-finite and satisfies
	\begin{align*}
		(1-2z) \Cgf_1''(z) - (3-z)\Cgf_1'(z) &= 0, &
		\Cgf_1(0)=1, \quad
		\Cgf'_1(0)=1.
	\end{align*}
	The closed-form formula for $\Cgf_1'(z)$, and the asymptotic behavior of the coefficients are given by
	\begin{align*}
		\Cgf_1'(z) &= \frac{e^{z/2}}{(1-2z)^{5/4}}, &
		\cgf_{1,n} &= \frac{e^{1/4}}{\Gamma(1/4)} \frac{ n! 2^{n+1}}{ n^{3/4} } \left( 1 + \LandauO\left(\frac{1}{n} \right) \right).
	\end{align*}
\end{theo}

\begin{proof}
	Summing the result of Corollary~\ref{coro:C1l} for $l \geq 1$, interchanging summation, differentiation, and finally integration gives
	\begin{align*}
		(1-2z) \Cgf_1'(z) - \Cgf_1(z) - (1-z) \left( (1-z) \Cgf_{1,0} \right)' + \int z \Cgf_1'(z) \, dz = 0.
	\end{align*}
	Due to the remaining integral we differentiate both sides once more and get
	\begin{align}
		(1-2z) \Cgf_1''(z) - (3-z)\Cgf_1'(z) - \left((1-z) \left( (1-z) \Cgf_{1,0} \right)' \right)' = 0. \label{eq:C1dfinite}
	\end{align}
	Inserting $\Cgf_{1,0}(z) = \Cgf_0(z) = \frac{1}{1-z}$ we get the claimed differential equation
	\begin{align*}
		(1-2z) \Cgf_1''(z) - (3-z)\Cgf_1'(z) &= 0
	\end{align*}
	It can be solved by separation of variables with respect to $\Cgf_1'(z)$.
	The asymptotic behavior of the coefficients follows then directly from this representation.
\end{proof}

\subsection{Compacted trees of right height at most 2}
\label{sec:C2}

We decompose $\Cc_2$ such that we get
\begin{align*}
	\Cgf_2(z) &= \sum_{\ell \geq 0} \Cgf_{2,\ell}(z), 
\end{align*}
where $\Cgf_{2,\ell}(z)$ is the exponential generating function of compacted trees of right height at most $2$ with exactly $\ell$ right edges in the spine going from level $0$ to level $1$.  Obviously, we have $C_{2,0}(z) = \frac{1}{1-z}$. 

In the sequel we will use the operator calculus introduced in Section~\ref{sec:operations}.

\begin{prop}
	\label{prop:C2l}
	The generating function of compacted trees with right height at most $2$, and exactly $\ell$ right edges from level $0$ to level $1$ in the spine is given by
	\begin{align*}
		\Cgf_{2,\ell}(z) &= \Cgf_{2,\ell,A}(z) + \Cgf_{2,\ell,B}(z),\\
		\Cgf_{2,\ell,A}(z) &= A(\Cgf_{2,\ell-1}(z)),\\
		\bar{D}_2(\Cgf_{2,\ell,B}(z)) &= \bar{H}_2({C}_{2,\ell-1}(z)),
	\end{align*}
	with the linear operators $A = S \cdot I \cdot S \cdot K$, $\bar{D}_2 = M_1 \cdot D \cdot S^{-1}$, $\bar{H}_2 = (H_1 - M_1) \cdot (D\cdot z + S \cdot K)$, and $M_1$ and $H_1$ are given by
	\begin{align}
			M_1 &= (1-2z) D^2 - (3-z) D, \label{eq:compM1}\\
			H_1 &= (1-z)^2 D^2 - 3(1-z) D + 1. \label{eq:compH1}
	\end{align}
\end{prop}

\begin{proof}
Using the same ideas as in the case of relaxed trees, we reduce
the number of levels by deleting the initial sequence, 
and moving the last sequence to the end of the next lower level, see Figure~\ref{fig:R21decomp}.
This produces a $\Cc_1$-like object with 
\begin{itemize}
	\item a new initial condition $\hat{\Cgf}_{2,0}(z)$ and
	\item the restriction of being non-empty. 
\end{itemize}

\begin{figure}[htb]
	\centering
	\includegraphics[width=0.33\textwidth]{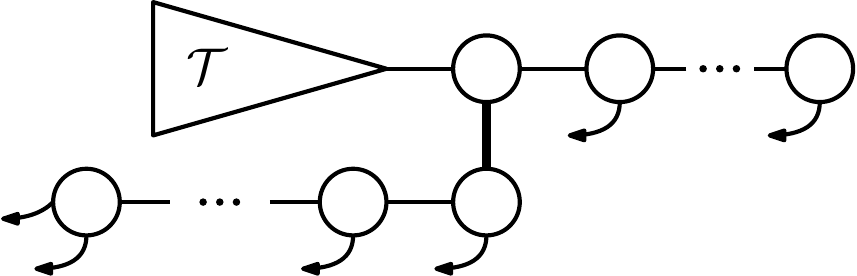}
	\qquad
	\includegraphics[width=0.4\textwidth]{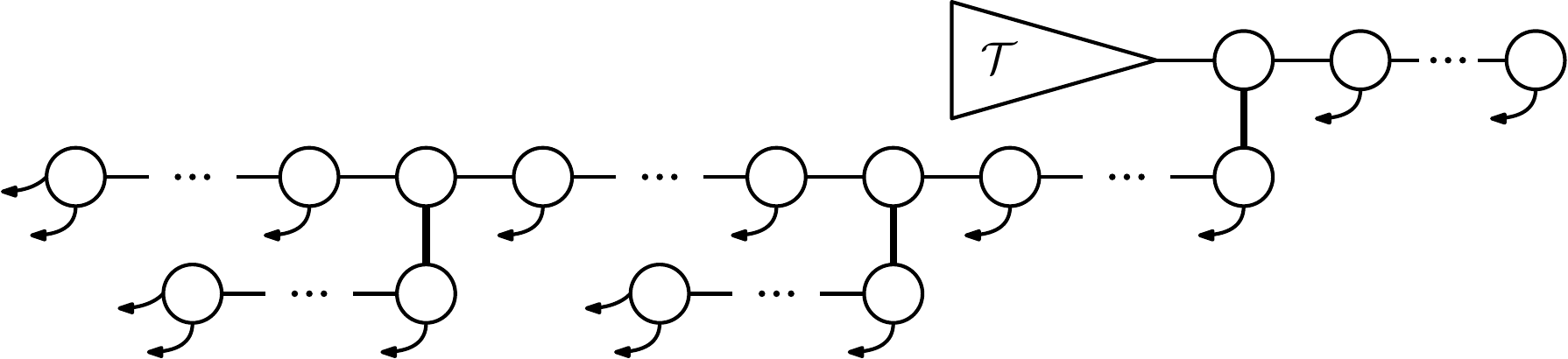}
	\caption{\small The $2$ possible cases for $\Cc_{2}$ instances: 
		Case (A) on the left, where level $2$ does not exist;
		and case (B) on the right, where level $2$ does exist.}
	\label{fig:C2AB}
\end{figure}

In contrast to the relaxed case of $\Rc_{2,1}$ we need to distinguish whether level $2$ exists or not
(Figure~\ref{fig:C2AB}). The different behaviors of single pointers and (double)
cherry pointers are responsible for these two cases. Let $\Cgf_{2,1} \in \Cc_{2,1}$,
$\hat{\Cgf}_{2,0}$ be the transformed object and $\Cgf_{2,0}$ be the part of level~$0$ located after the first right edge on level $0$ ($\Tc$ in Figure~\ref{fig:C2AB}).

\begin{figure}[htb]
	\centering
	\includegraphics[width=0.4\textwidth]{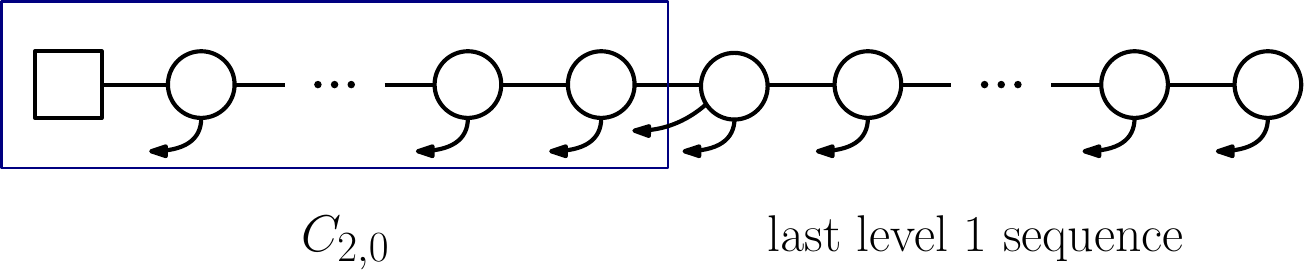}
	\caption{\small The new initial condition $\hat{\Cgf}_{2,0}(z)$. In case (A) the last sequence on level~$1$ cannot be empty, whereas in case (B) it may.}
	\label{fig:C20hat}
\end{figure}

\begin{enumerate}
	\item[(A)] 
		Let $C_{2,1,A}(z)$ be the generating function of compacted trees belonging to Case (A) and $\Cc_{2,1}$.
		In this case level $2$ does not exist
		(i.e.~the tree also belongs to $\Cc_{1}$). Then we need to have a cherry on level $1$,
		as this level is not allowed to be empty. This implies that the sequence 
		of $\hat{\Cgf}_{2,0}$ shown in Figure~\ref{fig:C20hat} cannot be empty.
		Then, due to the previous reasoning on relaxed trees (cf. Proposition~\ref{prop:R11}),
		and results on the trees in $\Cc_{1}$ (Corollary~\ref{coro:C1l}), we get the new initial condition for Case~(A): Instead of $C_{1,0}(z)$ in~\eqref{eq:C1dfinite} we must use
		\begin{align*}
			\hat{\Cgf}_{2,0,A}(z) 
				&:= \frac{1}{1-z}  K\left( \Cgf_{2,0}(z) \right)  
				 = \frac{z\left( z \Cgf_{2,0}(z)\right)' - \int z \Cgf_{2,0}'(z) \, dz}{1-z}.
		\end{align*}
		This implies
		\begin{align*}
			\Cgf_{2,1,A}(z) &= A(\Cgf_{2,0,A}(z)) := \frac{1}{1-z} \int \hat{\Cgf}_{2,0,A}(z) \, dz = \frac{1}{1-z} \int \frac{1}{1-z} K \left( \Cgf_{2,0,A}(z) \right) \, dz.
		\end{align*}
		The first factor $\frac{1}{1-z}$ corresponds to the initial sequence on level $0$, and the integral generates the level $0$ node of the distinguished right edge. In anticipation of the subsequent result, we introduced the operator $A$.
		
	\item[(B)] 
		Let $\Cgf_{2,1,B}(z)$ be the generating function of this case. 
		In this case level $2$ exists. 
		Then the last sequence of level $1$, and therefore the initial sequence of $\hat{\Cgf}_{2,0}$, is allowed to be empty, see Figure~\ref{fig:C20hat}. 
		This means that no cherry was lost during the transformation into an instance of $\Cc_{1}$
		as there is just one pointer pointing into $\Cgf_{2,0}$.
		Such a case is modeled by $(z \Cgf_{2,0}(z))'$.
		Combining it with the case of a non-empty sequence, we get the new initial condition of the case~(B):
		\begin{align*}
			\hat{\Cgf}_{2,0,B}(z) 
				&:= \left(z \Cgf_{2,0}(z) \right)' + \frac{1}{1-z} K\left( \Cgf_{2,0}(z) \right)   
				 = \frac{ \left(z \Cgf_{2,0}(z)\right)' - \int z \Cgf_{2,0}'(z) \, dz}{1-z}.
		\end{align*}

		The only difference to case (A) is the lack of the factor $z$ in front of $\left(z \Cgf_{2,0}(z)\right)'$.
		
		By assumption we have nodes on level $2$.
		This means that after the transformation into an instance of $\Cc_{1}$
		 we have nodes on level $1$. Let $\bar{\Cgf}_1(z)$ be the exponential generating function
		of compacted trees of right height at most $1$ with at least one node on level~$1$. Then
		\begin{align*}
			\bar{\Cgf}_1(z) &= \sum_{\ell \geq 1} \Cgf_{1,\ell}(z) = \Cgf_{1}(z) - \Cgf_{1,0}(z).
		\end{align*}
		The new operators, defined in \eqref{eq:compM1} and \eqref{eq:compH1}, fulfill the same tasks as the ones from Theorem~\ref{theo:diffoprelaxed} for relaxed trees. 
		From~\eqref{eq:C1dfinite} we infer that $\Cgf_{1}(z)$ satisfies $M_1(\Cgf_1(z)) = H_1(\Cgf_{0}(z))$. Thus, for $\bar{\Cgf}_1(z)$ we get the following differential equation:
		\begin{align*}
			M_1(\bar{\Cgf}_1(z)) = M_1(\Cgf_1(z) - \Cgf_{1,0}(z)) = H_1(\Cgf_{1,0}(z)) - M_1( \Cgf_{1,0}(z)) 
			.
		\end{align*}
		Then, the differential equation for $\Cgf_{2,1,B}(z)$ is given by
		\begin{align*}
			\bar{D}_2(\Cgf_{2,1,B}(z)) &:= M_1 \left( \left((1-z) \Cgf_{2,1,B}(z)\right)'\right) \\
			                           &= H_1(\hat{\Cgf}_{2,0,B}(z)) - M_1( \hat{\Cgf}_{2,0,B}(z) )
												 =: \bar{H}_2({C}_{2,0}(z)),
		\end{align*}
		because we are able to reuse~\eqref{eq:C1dfinite}, with the new initial condition $\hat{\Cgf}_{2,0,B}(z)$ instead of $C_{1,0}(z)$. Its solution is equal to $\left((1-z) \Cgf_{2,1,B}(z)\right)'$; 
		see the process of Proposition~\ref{prop:R21}. 
		The new differential operator is thus given by
		\begin{align*}
			\bar{D}_2(F) = (2z^2-3z+1) F''' - (z^2-10z+6)F'' - (2z-6) F'.
		\end{align*}	
\end{enumerate}

This process can now be continued recursively, like in Corollary~\ref{coro:R2l}.
In order to derive $\Cgf_{2,2}(z)$, we replace $\Cgf_{2,0}(z)$ by $\Cgf_{2,1}(z)$, and so on.
\end{proof}

Using the last result we are able to characterize compacted trees of right height at most~$2$. 

\begin{theo}
	\label{theo:C2}
	The exponential generating function of compacted trees of right height at most~$2$ is D-finite and satisfies
	\begin{align*}
		(z^2-3z+1) \Cgf_2'''(z) - (z^2-6z+6) \Cgf_2''(z) - (2z-3) \Cgf_2'(z) &= 0, \\
		\Cgf_2(0)=1,~\Cgf_2'(0)=1,~\Cgf_2''(0)&=3.
	\end{align*}
\end{theo}

\begin{proof}
	The generating function $\Cgf_2(z)$ is decomposed into three parts:
	\begin{align*}
		\Cgf_{2}(z) &= \Cgf_{2,0}(z) + \Cgf_{2,A}(z) + \Cgf_{2,B}(z),
	\end{align*}
	where $\Cgf_{2,A}(z) = \sum_{\ell \geq 0} \Cgf_{2,\ell, A}(z)$, $\Cgf_{2,B}(z) = \sum_{\ell \geq 0} \Cgf_{2,\ell, B}(z)$, and the initial values $\Cgf_{2,0,A}(z) = \Cgf_{2,0,B}(z) = 0$. 
	Summing the results of Proposition~\ref{prop:C2l} for $\ell \geq 1$ gives
	\begin{align*}
		\Cgf_{2,A}(z) &= A(\Cgf_{2}(z)),\\
		\bar{D}_2(\Cgf_{2,B}(z)) &= \bar{H}_2({C}_{2}(z)).
	\end{align*}
	Finally, we get
	\begin{align*}
		\bar{D}_2(\Cgf_{2}) &= \bar{D}_2(\Cgf_{2,0} + \Cgf_{2,A} + \Cgf_{2,B}) \\
		                       &= \bar{D}_2(\Cgf_{2,0}) + \bar{D}_2(A(\Cgf_{2})) + \bar{H}_2({\Cgf}_{2}),
	\end{align*}
	which gives the new differential operator $M_{2}$ and the inhomogeneity operator $H_2$:
	\begin{align}
		\label{eq:compM2operators}
		M_2(\Cgf_{2}) := \bar{D}_2(\Cgf_{2}) - \bar{D}_2(A(\Cgf_{2})) - \bar{H}_2({\Cgf}_{2}) = \bar{D}_2(\Cgf_{2,0}) =: H_2(\Cgf_{2,0}).
	\end{align}
	Note that in analogy to~\eqref{eq:relLHk}  from the relaxed case we have here
	\begin{align*}
		H_{2}  = M_1 \cdot D \cdot (1-z).
	\end{align*}
	The computation of $M_2$ is direct with a computer algebra system (like Maple or Sage).
\end{proof}

\subsection{Compacted trees of right height at most \texorpdfstring{$k$}{k}}

Analogous to the case of relaxed trees in Section~\ref{sec:relaxedheightk} the approach from the previous section can be generalized to an arbitrary bound $k \geq 2$ for the right height. 

We introduce linear differential operators $M_k$, $k \geq 0$ which describe all differential equations
constructed for $\Cgf_k(z)$. We use the same notation as in Section~\ref{sec:relaxedheightk}.
\begin{theo}[Differential operators]
	\label{theo:diffopcompacted}
	Let $(M_k)_{k \geq 0}$ be the family of differential operators given by
	\begin{align*}
		M_0 &= (1-z) D - 1,\\
		M_1 &= (1-2z) D^2 - (3-z)D, \\
		M_{k} &= M_{k-1} \cdot D - M_{k-2} \cdot \left( D^2 \cdot z - z D\right), \qquad k \geq 2. 
	\end{align*}
	Then the exponential generating function $\Cgf_k(z)$ of compacted binary trees with right height at most $k$ satisfies
	\begin{align*}
		M_k \cdot \Cgf_k &= 0.
	\end{align*}
\end{theo}

\begin{proof}
	The ideas are the same as the ones introduced in the proof of Theorem~\ref{theo:diffoprelaxed}:
	In an instance of $\Cc_k$ we cut at the first right edge in the spine from level $0$ to level $1$.
	Then, the same decomposition as in the case $k=2$ can be applied (like in Figure~\ref{fig:R21decomp}). 
	
	In particular, first generalizing the results of Propositition~\ref{prop:C2l} we get
	\begin{align*}
		\bar{D}_k &= M_{k-1} \cdot D \cdot S^{-1} \\
		\bar{H}_k &= \left( \bar{D}_{k-1} - M_{k-1} \right) (D \cdot z + S \cdot K ),
	\end{align*}
	which together with~\eqref{eq:compM2operators} implies on the level of operators
	\begin{align*}
		M_{k} &= \bar{D}_k - \bar{D}_k \cdot A - \bar{H}_k \\
				  &= M_{k-1} \cdot D - M_{k-2} \cdot \left(D^2 z - z D\right).
	\end{align*}
	Recall for the second equality that $D$ and $I$ are inverse to each other, and 
	that $S^{-1} = (1-z)$, and $K = zDz - IzD$.
\end{proof}

The first few differential equations are
\begin{equation*}
	(1-2z) \frac{d^2}{dz^2} \Cgf_1(z) + (z-3) \frac{d}{dz}\Cgf_1(z) = 0,
\end{equation*}
\begin{equation*}
	(z^2 - 3z + 1) \frac{d^{3}}{dz^{3}} \Cgf_2(z) - (z^2-6z+6) \frac{d^2}{dz^2} \Cgf_2(z) - (2z-3) \frac{d}{dz} \Cgf_2(z) = 0,
\end{equation*}
\begin{equation*}
	(3z^2-4z+1) \frac{d^{4}}{dz^{4}} \Cgf_3(z) - (4z^2-18z+10) \frac{d^{3}}{dz^{3}} \Cgf_3(z) + (z^2-12z+14) \frac{d^{2}}{dz^{2}}\Cgf_3(z) + (z-3) \frac{d}{dz} \Cgf_3(z) = 0.
\end{equation*}

\begin{theo}[Properties of $M_k$]
	\label{theo:Ekprop}
	The operator $M_k$ is a linear differential operator of order~$k+1$ satisfying
	\begin{align*}
		M_k &= m_{k,k}(z) D^{k+1} 
						+ m_{k,k-1}(z) D^{k} 
						+ \ldots
						+ m_{k,0}(z) D 
						+ m_{k,-1}(z) ,
	\end{align*}
	where the $m_{k,i}(z)$ are polynomials given by the following recurrence relation for $k \geq 2$
	\begin{align*}
		m_{k,-1}(z) &= 0,\\
		m_{k,0}(z) &= 
			\begin{cases}
				-2z+3, & \text{ for $k$ even,}\\
				z-3, & \text{ for $k$ odd,}
			\end{cases}\\
		m_{k,i}(z) &= m_{k-1,i-1}(z) + (i+1) m_{k-2,i}(z) \\
		           &\qquad + (z-i-2) m_{k-2,i-1}(z) - z m_{k-2,i-2}(z), \qquad 1 \leq i \leq k-1,\\
		m_{k,k}(z) &= m_{k-1,k-1}(z) - z m_{k-2,k-2}(z), \\
		m_{k,i}(z) &= 0, \qquad i > k.
	\end{align*}
	The initial polynomials are $m_{0,-1}(z) = -1$, $m_{0,0} = 1-z$, $m_{1,-1}=0$, $m_{1,0}=z-3$, and $m_{1,1}(z) = 1-2z$. The leading coefficients $m_{k,k}(z)$ are the same as $\ell_{k,k}(z)$ from the relaxed case.
\end{theo}

\begin{proof}
	The proof is analogous to the one of Theorem~\ref{theo:Dkprop}. We omit the tedious calculations.
\end{proof}

It may seem artificial to start the second index at $-1$. However, the corresponding polynomials are equal to $0$ except when $k=0$. Thus, we are actually dealing with a differential equation of order $k$ in $F'(z)$. Another advantage is that the leading polynomial $m_{k,k}(z)$, which is the same as the one in the relaxed case $\ell_{k,k}(z)$, has the same indices.

Following the approach used for relaxed trees,
we then need to reveal the structure of the indicial polynomial. 

In order to compute the value $\delta_1 = \lim_{z \to \zeta}(z-\zeta)a_1(z)$
(compare with~\eqref{eq:difeqmerogen}),
we need the following result on $m_{k,k-1}(z)$.

\begin{lemma}[Transformed $m_{k,k-1}(z)$]
	\label{lem:chebyshevmkk1}
	For the coefficient $m_{k,k-1}(z)$ of the operator $M_k$ from Theorem~\ref{theo:Ekprop},
	we get
	\begin{align*}
		m_{k,k-1}(z) &= z^{\frac{k+2}{2}} h_{k+2}\left(\frac{1}{2 \sqrt{z}}\right),
	\end{align*}
	where 
	\begin{align*}
		h_k(z) &= \frac{\left(k-3 - 2(k^2+k-2)z^2\right) T_k(z) + \left( 1 + 2(k-1)z^2 \right) U_k(z)}{2(z^2-1)},
	\end{align*}
	and $T_k(z)$ and $U_k(z)$ are the Chebyshev polynomials of first and second kind, respectively.
\end{lemma}

\begin{proof}
		From Theorem~\ref{theo:Ekprop} we get the recurrence relation of $m_{k,k-1}(z)$:
		\begin{align*}
			m_{k,k-1}(z) = m_{k-1,k-2}(z) - z m_{k-2,k-3}(z) + (z-k-1)m_{k-2,k-2}(z).
		\end{align*}
		Its structure is similar to the one of $m_{k,k}(z)$, but with an additional perturbation $(z-k-1)m_{k-2,k-2}(z)$. Transforming it in the same way as the one of $m_{k,k}(z)$ we get with 
		$$
			h_{k+2}(z) := (2z)^{k+2} m_{k,k-1}\left(\frac{1}{4z^2}\right),
		$$
		for $k \geq 0$ the recurrence
		\begin{align*}
			h_{k+2}(z) = 2z h_{k+1}(z) - h_{k}(z) + \left(1 - 4z^2 (k+1) \right) U_{k}(z).
		\end{align*}	
		From the theory of recurrences with constant coefficients (with respect to $k$) \cite[Chapter~$4$]{KauersPaule11} we get that the solution space is generated by 
		$
			U_k(z), 
			T_k(z), 
			k U_k(z), 
			k T_k(z), 
			k^2 U_k(z),
			k^2 T_k(z).
		$
		Making an ansatz and comparing coefficients gives the result.
\end{proof}

\begin{prop}
	\label{prop:indicialpolycompacted}
	Let $I_k(\alpha) = \alpha^{\underline{k+1}} + \delta_1 \alpha^{\underline{k}} + \cdots + \delta_{k+1}$ be the indicial polynomial  of the $k$-th differential equation, and let $\rho_k$ be the smallest real root of $m_{k,k}(z)$. Then, we have $\delta_i = 0$ for $i > 1$, and $\delta_1 = \frac{m_{k,k-1}(\rho_k)}{m'_{k,k}(\rho_k)}$. Furthermore, we have
	\begin{align*}
		\delta_1 &= \frac{k}{2} + 1 - \frac{1}{k+3} - \left(\frac{1}{4} - \frac{1}{k+3}\right)\frac{1}{\cos^2\left(\frac{\pi}{k+3}\right)}. 
		\end{align*}
	The indicial polynomial is given by $I_k(\alpha) = \alpha^{\underline{k}} (\alpha-(k-\delta_1))$.
\end{prop}

\begin{proof}
	The first results are analogous to the ones in Proposition~\ref{prop:indicialpoly}: First, because of Lemma~\ref{lem:lkkproperties} we have $\delta_i=0$ for $i > 0$. Second, the expression for $\delta_1$ is the same 
	as~\eqref{eq:delta1explicit}, and follows from de l'Hospital's rule. 	
	Thus, the indicial polynomial is given by $I_k(\alpha) = \alpha^{\underline{k+1}} + \delta_1\alpha^{\underline{k}}$.
	
	We start with two simplifications for the root $x_k = \cos(\frac{\pi}{k+1})$ of $U_k(z)$ when inserted into $T_k(z)$. By the explicit expression 
	$
		T_k(x) = \cos \left( k \arccos(x)\right), \text{ for } |x| \leq 1,
	$
	we get
	\begin{align*}
		T_{k}(x_{k}) &= - \cos\left(\frac{\pi}{k+1} \right) = - x_{k}, && \text{ and } &
		T_{k+1}(x_{k}) &= -1.
	\end{align*}

	First, we consider $m_{k,k-1}(z)$. By Lemma~\ref{lem:chebyshevmkk1} we directly get
	\begin{align*}
		m_{k,k-1}(\rho_k) &= -\rho_{k}^{\frac{k+2}{2}} \frac{(k-1)x_{k+2} - 2((k+2)^2+k) x_{k+2}^3}{2 (x_{k+2}^2-1)},
	\end{align*}
	where $\rho_k = \frac{1}{4x_{k+2}^2}$, and recall that $U_{k}(x_k) = 0$.
	
	Second, we consider the derivative of $m_{k,k}(z)$. Therefore, we use the following connection between Chebyshev polynomials of the first and second kind, see~\cite[Section~18.9]{NIST:DLMF}:
	\begin{align*}
		U'_k(z) = \frac{(k+1) T_{k+1}(z) - z U_k(z)}{z^2-1}. 
	\end{align*}	
	Thus, by Lemma~\ref{lem:chebyshev} we get 
	\begin{align*}
		m'_{k,k}(\rho_k) &= \frac{\rho_k^{\frac{k-1}{2}}}{4} \frac{k+3}{x_{k+2}^2-1}.
	\end{align*}
	Combining these results shows the claim.
\end{proof}

We arrive at our main result for compacted binary trees, Theorem~\ref{theo:compactedmain}.

\begin{proof}[Proof of Theorem~\ref{theo:compactedmain}]
	The proof follows the same lines as the one of Theorem~\ref{theo:relaxededmain}. In particular, the third case of Theorem~\ref{theo:odeasymptregsing} gives the asymptotic result, as $\delta_1$ is irrational for all $k \in \N$.
\end{proof}

In contrast to relaxed trees, the asymptotics of compacted trees involves in general an irrational critical exponent. In Table~\ref{tab:compcritexp} we list their first explicit values. 

\begin{table}[ht]
	\begin{center}
	\begin{tabular}{|c||c|c||c|c|c||c|c|c|}
		\hline $k$      & $r_k$ & $r_k \approx$ & $\kappa_k \approx$ & $\alpha_k$ & $\alpha_k \approx$ & $\gamma_k \approx$ & $\beta_k$ & $\beta_k \approx$ \\
		\hline\hline 
		$1$ & $2$ & $2.000$ & 
		$0.708$ & $-\frac{3}{4}$ & $-0.750$ & 
		$0.564$ & $-\frac{1}{2}$ & $-0.5$ 
		\\
		$2$ & $4 \cos(\frac{\pi}{5})^2$ & $2.618$ & 
		$0.561$ & $-\frac{6}{5}-\frac{1}{20\cos(\frac{\pi}{5})^2}$ & $-1.276$ & 
		$0.447$ & $-1$ & $-1.0$ 
		\\
		$3$ & $3$ & $3.000$ & 
		$0.605$ & $-\frac{16}{9}$ & $-1.778$ & 
		$0.493$ & $-\frac{3}{2}$ & $-1.5$ 
		\\
		$4$ & $4 \cos(\frac{\pi}{7})^2$ & $3.246$ & 
		$0.873$ & $-\frac{15}{7}-\frac{3}{28\cos(\frac{\pi}{7})^2}$ & $-2.275$ & 
		$0.726$ & $-2$ & $-2.0$ 
		\\
		$5$ & $4 \cos(\frac{\pi}{8})^2$ & $3.414$ &
		$1.625$ & $-\frac{21}{8}-\frac{1}{8\cos(\frac{\pi}{8})^2}$ & $-2.772$  & 
		$1.379$ & $-\frac{5}{2}$ & $-2.5$ 
		\\
		$6$ & $4 \cos(\frac{\pi}{9})^2$ & $3.532$ & 
		$3.782$ & $-\frac{28}{9}-\frac{5}{36\cos(\frac{\pi}{9})^2}$ & $-3.268$ & 
		$3.260$ & $-3$ & $-3.0$ 
		\\
		$7$ & $4 \cos(\frac{\pi}{10})^2$ & $3.618$ & 
		$10.708$ & $-\frac{18}{5}-\frac{3}{20\cos(\frac{\pi}{10})^2}$ & $-3.766$  & 
		$9.350$ & $-\frac{7}{2}$ & $-3.5$ 
		\\
		\hline
  \end{tabular}
\end{center}
\caption{\small The number $\cgf_{k,n}$ ($\rgf_{k,n}$) of compacted trees (respectively relaxed trees)
	with $n$ internal nodes and right height at most $k$ is asymptotically equal to $\kappa_k n! r_k^{n} n^{\alpha_k}$
	(resp.~$\gamma_k n! r_k^{n} n^{\beta_k}$) with $r_k = \rho_k^{-1}$.
}
\label{tab:compcritexp}
\end{table}

\section{Conclusion}
\label{sec:conclusion}

In this paper we solved the asymptotic counting problem for compacted and relaxed binary trees of bounded right height. In a compacted binary tree repeatedly occurring subtrees have been deleted and replaced by pointers to the first appearance, and hence every subtree is unique. By doing so, the tree structure is destroyed and replaced by a directed acyclic graph. In a relaxed binary tree the uniqueness condition of subtrees is omitted. 

The difficulty of this counting problem is founded in the compaction procedure. A compacted binary tree of size $n$, where the size is the number of internal nodes, arises from a binary tree whose size is between $n$ and $2^n$.
Our main results with regard to the general counting problem, are recurrence relations for compacted and relaxed binary trees in Theorem~\ref{theo:comprecursion} and Corollary~\ref{coro:relaxedrecurrence}, respectively.

Due to their superexponential growth of order $\Theta(n! \rho_k^{-n} n^{\alpha_k})$ with $\rho_k \approx 1/4$, exponential generating functions are the natural choice. Our second main contribution is the development of a calculus on such exponential generating functions modeling the structural properties of compacted trees in Section~\ref{sec:operations}. 

Resulting from these ideas, we were able to give our last main result: the derivation of ordinary differential equations for relaxed and compacted binary trees of bounded right height. The right height of a tree is the maximal number of right edges from the root to any leaf. Furthermore, we extracted the asymptotics by extending the theory of coefficient extractions of ordinary differential equations with polynomial coefficients in Theorem~\ref{theo:odeasymptregsing}. This yielded the sought asymptotics in Theorems~\ref{theo:relaxededmain} and \ref{theo:compactedmain}. 

Thereby we discovered quite \emph{exotic} families of enriched trees. The radii of convergence are in both cases algebraic numbers, and in the case of compacted trees, also the critical exponents are (compare Table~\ref{tab:compcritexp} for the first $7$ families). 
Note that our techniques do not directly give access to the constants $\kappa_k$ and $\gamma_k$. 
They can be numerically computed for any specific case from the respective differential equations from the basis of asymptotic solutions like given on page~\pageref{eq:relaxedsols} at the end of Section~\ref{sec:relaxed}.
For more details see~\cite{Mezzarobba11}. 

It remains an open problem to find the asymptotics of relaxed and compacted trees without any restrictions. For our methods it was crucial that the right height was bounded by a fixed value $k$. The limit $k \to \infty$ is therefore not computable. In particular, we showed that the radius of convergence $\rho_k$ converges to $1/4$. But the subexponential growth is of the shape $n^{- \lambda k}$ for $\lambda>0$. Thus, it would converge to $0$. Hence, the limits $n\to \infty$ and $k \to \infty$ are not interchangeable. 
However, in Corollary~\ref{coro:mainexpgrowth} we were able to show that the exponential growth of the number of relaxed and compacted binary trees is equal to $4^n$. This behavior remains a topic of future research.

Finally, it is interesting to compare the number of compacted trees to the number of relaxed trees
in Corollary~\ref{coro:compamongrelaxed}. We showed that their number is negligible for large $n$
and derived a precise quantitative result.

Many new questions arise after our analysis. It would be interesting to consider parameters such as their average height or average right height. Furthermore, these results gave us generating functions of a large family of DAGs which should allow a uniform random generation of such trees. Such results are interesting in computer science and the analysis of algorithms, as DAGs are efficient data structures and widely-used. Among other things, new algorithms need to be tested on very large and non-trivial elements of an efficiently computable class of DAGs.

\addcontentsline{toc}{section}{Acknowledgments}
\section*{Acknowledgments}
\label{sec:ack}
We would like to thank Christian Krattenthaler for pointing out the connection with the Chebyshev polynomials, 
and Marc Mezzarobba for computing the numerical values of the first few constants $\kappa_k$ and $\gamma_k$ to high precision. 
The authors also thank the anonymous referees for their comments and suggested improvements.

\addcontentsline{toc}{section}{References}
\bibliographystyle{abbrv}
\bibliography{mybib}

\begin{thebibliography}{10}

\bibitem{AbramowitzStegun1964}
M.~Abramowitz and I.~A. Stegun.
\newblock {\em Handbook of {M}athematical {F}unctions {W}ith {F}ormulas,
  {G}raphs, and {M}athematical {T}ables}, volume~55 of {\em National Bureau of
  Standards Applied Mathematics Series}.
\newblock For sale by the Superintendent of Documents, U.S. Government Printing
  Office, Washington, D.C., 1964.

\bibitem{aho1986compilers}
A.~V. Aho, R.~Sethi, and J.~D. Ullman.
\newblock {\em Compilers, Principles, Techniques}.
\newblock Addison-Wesley, Boston, 1986.

\bibitem{ADK05}
C.~Akkan, A.~Drexl, and A.~Kimms.
\newblock Generating two-terminal directed acyclic graphs with a given
  complexity index by constraint logic programming.
\newblock {\em J. Log. Algebr. Program.}, 62(1):1--39, 2005.

\bibitem{BRRW86}
E.~A. Bender, L.~B. Richmond, R.~W. Robinson, and N.~C. Wormald.
\newblock The asymptotic number of acyclic digraphs. {I}.
\newblock {\em Combinatorica}, 6(1):15--22, 1986.

\bibitem{BR88}
E.~A. Bender and R.~W. Robinson.
\newblock The asymptotic number of acyclic digraphs. {II}.
\newblock {\em J. Combin. Theory Ser. B}, 44(3):363--369, 1988.

\bibitem{BDGP17}
O.~Bodini, M.~Dien, A.~Genitrini, and F.~Peschanski.
\newblock The ordered and colored products in analytic combinatorics:
  Application to the quantitative study of synchronizations in concurrent
  processes.
\newblock In {\em Proceedings of the Meeting on Analytic Algorithmics and
  Combinatorics}, ANALCO'17, 2017.

\bibitem{BGG11}
O.~Bodini, D.~Gardy, and B.~Gittenberger.
\newblock Lambda terms of bounded unary height.
\newblock In {\em Proceedings of the Meeting on Analytic Algorithmics and
  Combinatorics}, ANALCO '11, pages 23--32, Philadelphia, PA, USA, 2011.
  Society for Industrial and Applied Mathematics.

\bibitem{BGGG17}
O.~Bodini, D.~Gardy, B.~Gittenberger, and Z.~Go\l\k{e}biewski.
\newblock On the number of unary-binary tree-like structures with restrictions
  on the unary height.
\newblock {\em Ann. Comb.}, 22(1):45--91, 2018.

\bibitem{BGGJ13}
O.~Bodini, D.~Gardy, B.~Gittenberger, and A.~Jacquot.
\newblock Enumeration of generalized {BCI} lambda-terms.
\newblock {\em Electr. J. Comb.}, 20(4):P30, 2013.

\bibitem{bousquet2015xml}
M.~Bousquet-M{\'e}lou, M.~Lohrey, S.~Maneth, and E.~Noeth.
\newblock {XML} compression via directed acyclic graphs.
\newblock {\em Theory of Computing Systems}, 57(4):1322--1371, 2015.

\bibitem{CD12}
V.~Carnino and S.~De~Felice.
\newblock Sampling different kinds of acyclic automata using {M}arkov chains.
\newblock {\em Theoret. Comput. Sci.}, 450:31--42, 2012.

\bibitem{NIST:DLMF}
{\it NIST Digital Library of Mathematical Functions}.
\newblock http://dlmf.nist.gov/, Release 1.0.13 of 2016-09-16.
\newblock F.~W.~J. Olver, A.~B. {Olde Daalhuis}, D.~W. Lozier, B.~I. Schneider,
  R.~F. Boisvert, C.~W. Clark, B.~R. Miller and B.~V. Saunders, eds.

\bibitem{DowneySethiTarjan1980variations}
P.~J. Downey, R.~Sethi, and R.~E. Tarjan.
\newblock Variations on the common subexpression problem.
\newblock {\em J. Assoc. Comput. Mach.}, 27(4):758--771, 1980.

\bibitem{flaj09}
P.~Flajolet and R.~Sedgewick.
\newblock {\em Analytic Combinatorics}.
\newblock Cambridge University Press, 2009.

\bibitem{flss90}
P.~Flajolet, P.~Sipala, and J.-M. Steyaert.
\newblock Analytic variations on the common subexpression problem.
\newblock In {\em Automata, languages and programming ({C}oventry, 1990)},
  volume 443 of {\em Lecture Notes in Comput. Sci.}, pages 220--234. Springer,
  New York, 1990.

\bibitem{GS05}
P.~J. Grabner and B.~Steinsky.
\newblock Asymptotic behaviour of the poles of a special generating function
  for acyclic digraphs.
\newblock {\em Aequationes Math.}, 70(3):268--278, 2005.

\bibitem{henr77}
P.~Henrici.
\newblock {\em Applied and {C}omputational {C}omplex {A}nalysis. {V}ol. 2}.
\newblock Wiley Classics Library. John Wiley \& Sons, Inc., New York, 1991.
\newblock Special functions---integral transforms---asymptotics---continued
  fractions, Reprint of the 1977 original, A Wiley-Interscience Publication.

\bibitem{ince44}
E.~L. Ince.
\newblock {\em Ordinary {D}ifferential {E}quations}.
\newblock Dover Publications, New York, 1944.

\bibitem{KauersPaule11}
M.~Kauers and P.~Paule.
\newblock {\em The {C}oncrete {T}etrahedron}.
\newblock Texts and Monographs in Symbolic Computation. SpringerWienNewYork,
  Vienna, 2011.
\newblock Symbolic sums, recurrence equations, generating functions, asymptotic
  estimates.

\bibitem{liskovets2006exact}
V.~A. Liskovets.
\newblock Exact enumeration of acyclic deterministic automata.
\newblock {\em Discrete Applied Mathematics}, 154(3):537--551, 2006.

\bibitem{MSW15}
C.~McDiarmid, C.~Semple, and D.~Welsh.
\newblock Counting phylogenetic networks.
\newblock {\em Ann. Comb.}, 19(1):205--224, 2015.

\bibitem{McKay89}
B.~D. McKay.
\newblock On the shape of a random acyclic digraph.
\newblock {\em Math. Proc. Cambridge Philos. Soc.}, 106(3):459--465, 1989.

\bibitem{MDB01}
G.~Melan{\c{c}}on, I.~Dutour, and M.~Bousquet-M{\'e}lou.
\newblock Random generation of directed acyclic graphs.
\newblock In {\em Comb01---{E}uroconference on {C}ombinatorics, {G}raph
  {T}heory and {A}pplications}, volume~10 of {\em Electron. Notes Discrete
  Math.}, page 6 pp. (electronic). Elsevier, Amsterdam, 2001.

\bibitem{MP04}
G.~Melan{\c{c}}on and F.~Philippe.
\newblock Generating connected acyclic digraphs uniformly at random.
\newblock {\em Inform. Process. Lett.}, 90(4):209--213, 2004.

\bibitem{Mezzarobba11}
M.~Mezzarobba.
\newblock {\em Autour de l'{\'e}valuation num{\'e}rique des fonctions
  D-finies}.
\newblock PhD thesis, {\'E}cole polytechnique, France, Palaiseau, November
  2011.

\bibitem{Polya37}
G.~P{\'o}lya.
\newblock Kombinatorische {A}nzahlbestimmungen f{\"u}r {G}ruppen, {G}raphen und
  chemische {V}erbindungen.
\newblock {\em Acta Mathematica}, 68(1):145--254, 1937.

\bibitem{RalaivaosaonaWagner2015Repeated}
D.~Ralaivaosaona and S.~Wagner.
\newblock Repeated fringe subtrees in random rooted trees.
\newblock In {\em 2015 {P}roceedings of the {T}welfth {W}orkshop on {A}nalytic
  {A}lgorithmics and {C}ombinatorics ({ANALCO})}, pages 78--88. SIAM,
  Philadelphia, PA, 2015.

\bibitem{Robinson70}
R.~W. Robinson.
\newblock Enumeration of acyclic digraphs.
\newblock In {\em Proc. {S}econd {C}hapel {H}ill {C}onf. on {C}ombinatorial
  {M}athematics and its {A}pplications ({U}niv. {N}orth {C}arolina, {C}hapel
  {H}ill, {N}.{C}., 1970)}, pages 391--399. Univ. North Carolina, Chapel Hill,
  N.C., 1970.

\bibitem{robinson1973dags}
R.~W. Robinson.
\newblock Counting labeled acyclic digraphs.
\newblock In {\em New directions in the theory of graphs ({P}roc. {T}hird {A}nn
  {A}rbor {C}onf., {U}niv. {M}ichigan, {A}nn {A}rbor, {M}ich., 1971)}, pages
  239--273. Academic Press, New York, 1973.

\bibitem{Robinson77}
R.~W. Robinson.
\newblock Counting unlabeled acyclic digraphs.
\newblock In {\em Combinatorial {M}athematics, {V} ({P}roc. {F}ifth {A}ustral.
  {C}onf., {R}oy. {M}elbourne {I}nst. {T}ech., {M}elbourne, 1976)}, pages
  28--43. Lecture Notes in Math., Vol. 622. Springer, Berlin, 1977.

\bibitem{SchlesingerI1968}
L.~Schlesinger.
\newblock {\em Handbuch der {T}heorie der linearen {D}ifferentialgleichungen.
  {I}n zwei {B}\"anden, {B}and {I}}.
\newblock Reprint. Bibliotheca Mathematica Teubneriana, Band 30. Johnson
  Reprint Corp., New York-London, 1968.

\bibitem{stan99}
R.~P. Stanley.
\newblock {\em Enumerative {C}ombinatorics. {V}ol. 2}, volume~62 of {\em
  Cambridge Studies in Advanced Mathematics}.
\newblock Cambridge University Press, Cambridge, 1999.
\newblock With a foreword by Gian-Carlo Rota and appendix 1 by Sergey Fomin.

\bibitem{St03}
B.~Steinsky.
\newblock Enumeration of labelled chain graphs and labelled essential directed
  acyclic graphs.
\newblock {\em Discrete Math.}, 270(1-3):267--278, 2003.

\bibitem{St04}
B.~Steinsky.
\newblock Asymptotic behaviour of the number of labelled essential acyclic
  digraphs and labelled chain graphs.
\newblock {\em Graphs Combin.}, 20(3):399--411, 2004.

\bibitem{Wa13}
S.~Wagner.
\newblock Asymptotic enumeration of extensional acyclic digraphs.
\newblock {\em Algorithmica}, 66(4):829--847, 2013.

\bibitem{Wallner19R1}
M.~Wallner.
\newblock A bijection of plane increasing trees with relaxed binary trees of
  right height at most one.
\newblock {\em Theoretical Computer Science}, 755:1--12, 2019.

\bibitem{waso87}
W.~Wasow.
\newblock {\em Asymptotic {E}xpansions for {O}rdinary {D}ifferential
  {E}quations}.
\newblock Dover Publications, Inc., New York, 1987.
\newblock Reprint of the 1976 edition.

\end{thebibliography}
\label{sec:biblio}

\end{document}